\documentclass[3p, numbers]{elsarticle}
\usepackage{rsfso} 
\usepackage{kpfonts} 
\usepackage{amsthm}
\usepackage{mathtools}
\usepackage{bbm}
\usepackage{tikz-cd}
\usepackage{hyperref}

\theoremstyle{plain}
\newtheorem{theorem}{Theorem}[section]
\newtheorem{corollary}{Corollary}[section]
\newtheorem{lemma}{Lemma}[section]
\newtheorem{claim}{Claim}[section]
\newtheorem{proposition}{Proposition}[section]

\theoremstyle{definition}
\newtheorem{definition}{Definition}[section]

\newtheorem{example}{Example}[section]
\newtheorem{problem}{Question}

\newtheorem{remark}{Remark}[section]

\numberwithin{equation}{section}

\renewcommand{\leq}{\leqslant}
\renewcommand{\geq}{\geqslant}
\renewcommand{\preceq}{\preccurlyeq}

\newcommand{\iv}[1]{{#1}^{-1}}
\newcommand{\e}{\mathrm{e}}

\newcommand{\jd}{I}
\newcommand{\sst}{\subseteq}
\newcommand{\ksg}{\mathfrak{k}}
\newcommand{\Ksg}{\mathfrak{C}}
\newcommand{\Ideal}{\mathfrak{I}}
\newcommand{\p}{\mathfrak{p}}
\newcommand{\q}{\mathfrak{q}}
\newcommand{\m}{\mathfrak{m}}

\renewcommand{\P}{{{\sf P}}}

\newcommand{\V}{{{\sf V}}}
\renewcommand{\L}{{{\sf L}}}
\newcommand{\Rep}{{{\sf R}}}
\newcommand{\A}{{{\sf A}}}

\newcommand{\Z}{\mathbb{Z}}
\newcommand{\N}{\mathbb{N}}

\newcommand{\Bord}[2]{\mathscr{B}_{#2}\,({#1})}
\newcommand{\Pord}[2]{\mathscr{P}_{\,\,#2}\,({#1})}

\DeclareMathOperator{\RPord}{{\mathscr{P}}}
\DeclareMathOperator{\BOrd}{{\mathscr{B}}}
\DeclareMathOperator{\Rord}{{\mathscr{R}}}
\DeclareMathOperator{\Ord}{{\mathscr{O}}}
\newcommand{\Pa}{\mathbb{P}}

\newcommand{\Xa}{\mathbb{X}}
\DeclareMathOperator{\Aut}{Aut}

\newcommand{\Connp}{{\rm Conv}_{\rm p}^*}
\newcommand{\Conn}{{\rm Conv}^*}
\DeclareMathOperator{\Con}{{\rm Conv}}
\DeclareMathOperator{\Conp}{{\rm Conv}_{\rm p}}
\DeclareMathOperator{\Spec}{{\rm Spec}}
\newcommand{\Specn}{{\rm Spec}^*}
\DeclareMathOperator{\Min}{{\rm Min}}
\DeclareMathOperator{\Qin}{{\rm Qin}}

\DeclareMathOperator{\Pol}{{\rm Pol}}
\DeclareMathOperator{\Polp}{{\rm Pol}_{\rm p}}
\newcommand{\Supp}{\mathbb{S}}
\newcommand{\Van}{\mathbb{V}}

\newcommand{\Free}[2]{F_{#2}\,({#1})}
\newcommand{\F}[1]{F_{#1}}

\newcommand\nbd[1]{\protect\nobreakdash#1\hspace{0pt}}

\begin{document}
\begin{frontmatter}
\journal{ArXiv}
\title{Orders on groups, and spectral spaces of lattice-groups}
\author[ac]{Almudena~Colacito}
\ead{almudena.colacito@math.unibe.ch}
\author[vm]{Vincenzo~Marra}
\ead{vincenzo.marra@unimi.it}
\address[ac]{Mathematisches Institut, Universit\"at Bern, Alpeneggstrasse 22, 3012 Bern, Switzerland}
\address[vm]{Dipartimento di Matematica ``Federigo Enriques'', Universit\`a degli Studi di Milano, via Cesare Saldini 50, 20133 Milano, Italy}

\begin{abstract}Extending  pioneering work by Weinberg, Conrad, McCleary, and others, we provide a systematic way of relating  spaces of right orders on a partially ordered group, on the one hand, and  spectral spaces of  free lattice-ordered groups, on the other. The aim of the theory is to pave the way for further fruitful interactions between the study of right orders on groups and that of lattice-groups. Special attention is paid to the important case of orders on groups. 
\end{abstract}

\end{frontmatter}

\section{Introduction}\label{s:intro}
A \emph{right order} on a group $G$ is a total order $\leq$ on $G$ such that $x\leq y$ implies $xt\leq yt$, for all $x,y,t\in G$. There are at least two distinct motivations for studying such orders on groups. First, a countable group  admits a right order if, and only if, it acts faithfully on the real line by orientation-preserving homeomorphisms.  This, as far as we know, is folklore; see~\cite[Theorem 6.8]{Ghys2001} for a proof. (While we use right orders in this paper, other authors prefer left orders as in~\cite{Ghys2001}, the difference being immaterial.) The result indicates that orders on groups  play a r\^ole  in topological dynamics. For more on this, besides Section 6.5 of Ghys' beautiful survey~\cite{Ghys2001}, see the research monograph ~\cite{CR2016}. The reader can also consult~\cite{DehornoyEtAl2008}, and especially Chapter 14, for the important case of orders on braid groups. 

Second, the collection of all right orders on a group $G$ leads to a representation of  the free lattice\nbd{-}ordered group (or free \emph{$\ell$-group})  generated by~$G$. This classical theorem is due to Conrad~\cite{Conrad1970}  (see also~\cite{MR582096} for a generalisation), who extended a previous result of Weinberg~\cite{Weinberg1963} for the Abelian case. Right orders have since been central to the theory of $\ell$-groups; standard references for the latter are~\cite{BigardEtAl1977} and~\cite{Darnel95}. All lattice-group notions  mentioned in the rest of this introduction are defined in the subsequent sections.

In 2004, Sikora topologised the set of right orders on a group~\cite{Sikora2004}, and studied the resulting space;  see again~\cite{CR2016}.  In this paper we show, {\it inter alia}, that Sikora's space arises naturally from the study of lattice\nbd{-}groups, as the minimal spectrum of the $\ell$-group freely generated by the group at hand. More generally, by replacing right orders with right {\em pre}-orders---pre\nbd{-}orders that are invariant under group multiplication on the right---we provide a systematic, structural account of the relationship between (always total) right pre-orders on a group $G$ and prime subgroups of the $\ell$-group $\Free{G}{}$ freely generated by $G$. One  often restricts attention to special classes of right pre-orders on a group. For example, one may be interested in studying  \emph{orders} (sometimes called \emph{bi-orders}) on a group $G$, that is, total orders on $G$ that are invariant under group multiplication both on the right and on the left. In that case one  needs to replace $\Free{G}{}$ with the free \emph{representable} $\ell$-group generated by $G$. In full generality, fix an arbitrary variety (=equationally definable class) {\V} of $\ell$-groups,  a partially ordered group $G$, and write $\Free{G}{\V}$ for the free $\ell$-group generated by $G$ in $\V$. The results in this paper provide a theory of the connection between the prime subgroups  of $\Free{G}{\V}$, and a corresponding class of right pre-orders on $G$, where each pre-order in the class is required to  extend the partial order of $G$.

The connection we exhibit and study here has been previously identified in its basic form by McCleary in his paper on representations of free lattice-groups by ordered permutation groups, cf.\ \cite[Lemma 16]{McCleary1985}. There, McCleary considers a free group $G$ and constructs a bijection between right orders on $G$ and minimal prime subgroups of $\Free{G}{}$. This  paper may be viewed as a generalisation and extension of McCleary's result. Let us also mention~\cite{Clay2002}, where the  author acknowledges  McCleary's work as a source for his own correspondence  between closures of orbits (under the natural action of $G$) in the space of right orders on a group $G$, and kernels of certain maps from $\Free{G}{}$. 

In Section~\ref{s:main} we describe our construction and state our main correspondence result, Theorem~\ref{t:main}. We also state Corollary~\ref{c:top}, establishing that the spaces of right pre-orders corresponding to varieties of lattice\nbd{-}groups are always \emph{completely normal generalised spectral spaces}. Here and throughout the paper, `spectral space' is meant in the sense of Hochster~\cite{Hochster69}. \emph{Generalised} spectral spaces are the not necessarily compact versions of spectral spaces. Sections~\ref{s:orderiso} and~\ref{s:homeo} are devoted to a proof of Theorem~\ref{t:main}. In Section~\ref{s:spec} we use ideas from  Stone duality to obtain Corollary~\ref{c:top}.

The construction leading to Theorem~\ref{t:main} shows that to each variety of $\ell$-groups there remains associated a class of right pre-orders on groups.  The class corresponding to the variety of all $\ell$-groups is just that of all possible right pre-orders. In Section~\ref{s:var} we specialise our main result to the  varieties of representable and Abelian $\ell$-groups, thereby characterising the right pre-orders that correspond to these two varieties; see Theorems~\ref{t:reptrans2} and~\ref{t:reptrans3}.

In Theorem~\ref{t:minorders} we show that the usual spaces of right orders, orders, and orders on Abelian groups are recovered in our framework as  the subspaces of inclusion-minimal right pre-orders of the appropriate class in each case.  Through the bijection of Theorem~\ref{t:main}, such subspaces correspond precisely to minimal spectra of the free lattice-group over the given group in the appropriate variety (Corollary~\ref{c:minetag}). This leads us in Section~\ref{s:min} to revisit the much-studied minimal spectrum of an $\ell$-group, with special attention to its possible compactness. We obtain a general algebraic compactness criterion  in Corollary~\ref{c:cbapolars}. We also prove as Theorem~\ref{t:compactin3vars} that in the varieties all $\ell$-groups, of representable $\ell$-groups, and of Abelian $\ell$-groups, the free lattice-group generated by any partially ordered group has compact minimal spectrum. The spaces of right orders or orders on a partially ordered group $G$ are shown to be the dual Stone spaces of the Boolean algebra of principal polars of the $\ell$-group freely generated by $G$ in the appropriate variety; this is the content of Theorem~\ref{t:asbooleanspaces}, which also encompasses the analogous result for Abelian groups.
 
In the final Section~\ref{s:rep} we focus on orders on groups. We have seen in Section~\ref{s:var} how to specialise our main result (Theorem~\ref{t:main}) to any variety of $\ell$-groups, and thus in particular to the representable one. However, our last Theorem~\ref{t:main_rep} shows that,  in studying orders on a partially ordered group~$G$,   one really ought to look at \emph{prime ideals} (=normal prime subgroups) of the free representable $\ell$\nbd{-}group generated by $G$, as opposed to its prime subgroups. To rephrase on the algebraic side: in varieties of  representable $\ell$\nbd{-}groups the notion of prime subgroup---which is indispensable in the study of general $\ell$-groups---should be done away with and replaced by the notion of prime ideal.

\section{Main construction and result}\label{s:main}
The set of natural numbers is $\N\coloneqq\{\, 0,1,2,\dots \,\}$. Throughout, by `lattice' we mean `partially ordered set such that infima ($\wedge$) and suprema ($\vee$) of finite nonempty subsets exist'; thus, lattices do not necessarily have maxima or minima, and even when they do, lattice homomorphisms need not preserve them. By a {\em partially ordered group} we mean a group $G$ equipped with a partial order $\leq$ compatible with the group operation, that is, from $a \leq b$ we can conclude $sat \leq sbt$, for every $a, b, t, s \in G$. We write $\P$ for the category of partially ordered groups and their positive (equivalently, order-preserving) group homomorphisms. Recall that a positive group homomorphism is an {\em order embedding} if it is injective, and order reflecting on the range. The {\em positive cone} of an object $G$ of $\P$ is $G^+ := \{\, a \in G \mid a \geq \e \,\}$, where `$\e$' is our notation for the identity element of a group. A {\em lattice-ordered group} $H$ (briefly, \emph{$\ell$-group}) is a partially ordered group whose partial order is a lattice order. By `$\ell$-subgroup' we mean `sublattice subgroup of $H$'. Recall that the set of order-preserving bijections $\Aut{\Omega}$ of any chain (=totally ordered set) $\Omega$ can be made into an $\ell$-group, with group operation $f \cdot g$ standing for `$g$ after $f$', and point-wise lattice operations. 

A binary relation ${\preceq} \subseteq S \times S$ on a set $S$ is a {\em pre-order} if it is reflexive, transitive, and total. `Total' means: for any $a,b\in S$, either $a\preceq b$ or $b\preceq a$ (or both). We establish the convention, valid throughout the paper, that `pre-order' means `\emph{proper} pre-order'; i.e., we assume ${\preceq} \neq S\times S$. This will be consistent with the standard assumption that prime subgroups are proper.

If $\preceq$ is a pre-order on a group $G$, then $\preceq$ is {\em right invariant} (respectively, {\em left invariant}) if for all $a, b, t \in G$, whenever $a \preceq b$ then $at \preceq bt$ (respectively, $ta \preceq tb$). For a partially ordered group $G$, a {\em right pre-order on $G$} is a right-invariant pre-order on $G$ that extends the partial order on $G$; a {\em pre-order on $G$} is a left-invariant right pre-order on $G$. A \emph{proper} submonoid $C \subsetneq G$ is a  {\em pre-cone} of $G$ if $G = C \cup \iv{C}$ and $G^+ \sst C$. We set 
\[
\RPord{(G)} \coloneqq \{\, C \sst G \mid C\text{ is a pre-cone of }G \,\}.
\]

\noindent For a partially ordered group $G$, the set of right pre-orders on $G$ is a poset under inclusion, and similarly, $\RPord{(G)}$ is partially ordered by inclusion. It is elementary that these two posets are isomorphic via the map that associates to $C \in \RPord{(G)}$ the relation: $a \preceq_C b$ if, and only if, $b\iv{a} \in C$. The inverse of this bijection sends a right pre-order $\preceq$ to its {\em positive cone} $C \coloneqq \{\, a \in G \mid \e \preceq a \,\}$. This isomorphism restricts to one between pre\nbd{-}orders on $G$ and those pre-cones of $G$ that are {\em normal}, i.e., closed under group conjugation; we denote the subposet of such pre-cones by $\BOrd{(G)}$. In the rest of this paper we deliberately confuse a right pre-order $\preceq$ on $G$ with its associated pre-cone $C$. We write $a \prec b$ to mean $a \preceq b$ and $b \npreccurlyeq a$.

The following definition provides the natural extension to right pre-orders of a construction that is standard for right orders---and of course, \textit{mutatis mutandis}, for unordered groups.
 We write $G_C$ for the partially ordered group $G$ equipped with the right pre-order $C\supseteq G^+$. Then $C$ induces an equivalence relation $\equiv_C$ on $G$ by: $a \equiv_C b$ if, and only if, $a \preceq_C b$ and $b \preceq_C a$. 
We write $[a]$ for the equivalence class of $a \in G$, where $C$ is understood from context. 
The quotient set of $G$ modulo $\equiv_C$, which we denote $\Omega_C$, is totally ordered by: $[a] \leq_C [b]$ if, and only if, $a \preceq_C b$. It is elementary that the  map  
\begin{align}\nonumber
G &\xrightarrow{ \ \rho_{C} \ } \Aut{\Omega_C} 
\end{align}
defined by $\rho_C(a)([b])=[ba]$ for $a,b\in G$ is a positive group homomorphism. Its image $\rho_C[G]$ is usually not an $\ell$-subgroup of $\Aut{\Omega_C}$. 
\begin{definition}[Right regular representation]\label{def:rrr}
Let $G$ be a partially ordered group, and let $C\supseteq G^+$ be a right pre-order on $G$. We denote by $H_C$ the $\ell$\nbd{-}subgroup of $\Aut{\Omega_C}$ generated by $\rho_C[G]$. The map 
\begin{align}\label{eq:Rrr}
G &\xrightarrow{ \ R_{C} \ } H_C \\
a &\xmapsto{ \hspace{0.53cm} } R_C(a)\colon [b] \mapsto [ba]\nonumber
\end{align}
is called the \emph{right regular representation} of $G_C$.
\end{definition}

\begin{remark}
We consistently use the notation $H_C$ throughout the paper as in the previous definition.
\end{remark}

\begin{definition}[Right pre-orders of a variety]\label{def:vpre} For any partially ordered group $G$ and any variety $\V$ of $\ell$\nbd{-}groups, we write $\Pord{G}{\V}$ for the set of right pre-orders $C \in \RPord{(G)}$ such that $H_C \in \V$. Further, we write $\Bord{G}{\V}$ for the subset of $\Pord{G}{\V}$ consisting of pre-orders.  
\end{definition}

\noindent Thus, writing $\L$ for the variety of all $\ell$-groups, $\Pord{G}{\L} = \RPord{(G)}$. Observe that~$\Pord{G}{\V}$ may well be empty even when $G$ is non-trivial.

We identify any variety $\V$ of $\ell$-groups with the full subcategory of the category of $\ell$-groups whose objects are the $\ell$-groups in $\V$.
Let us write $P \colon \V \longrightarrow \P$
for the faithful functor that takes an $\ell$-group $H$ in  $\V$ to $H$ itself regarded as a partially ordered group. By general considerations $P$ has a left adjoint $\F{\V} \colon \P\longrightarrow \V$, in symbols, $\F{\V}\dashv P$. For $G$ a partially ordered group, $\Free{G}{\V}$ is the $\ell$-group \emph{free over $G$ in $\V$}, or \emph{freely generated  by $G$ in $\V$}.
The component at $G$ of the unit of the adjunction $\F{\V} \dashv P$, written
\begin{equation}\label{eq:unitadj}
G \xrightarrow{\ \eta \ } \Free{G}{\V},
\end{equation}
is characterised by the following universal property:
{\it For each  positive homomorphism $p\colon G\to H$, with $H$ an $\ell$-group in $\V$, there is exactly one $\ell$\nbd{-}homomorphism $h\colon \Free{G}{\V}\to H$ such that $h\circ \eta=p$, i.e., such that the following diagram 
\begin{equation}\label{eq:universal}
\begin{tikzcd}
G \arrow[rd, "p"'] \arrow[r, "\eta"] & \Free{G}{\V} \arrow[d, dashed,  "h", "!"']\\
& H
\end{tikzcd}
\end{equation}
commutes}. We write $\Free{G}{}$ for $\Free{G}{\L}$. (Here and elsewhere we adopt the style common in algebra of omitting forgetful functors---$P$, for the case in point---unless clarity requires otherwise. Also, as in the above, we write $\eta$ in place of $\eta_G$ for the component of the unit, $G$ being understood.)
\begin{remark}
Bigard, Keimel, and Wolfenstein~\cite[Appendice A.2]{BigardEtAl1977}, and similarly, Conrad~\cite{Conrad1970}, distinguish between the \emph{universal} $\ell$-group on $G$, and the free such $\ell$-group, which in their terminology has the further property that the universal arrow $\eta \colon G\to F(G)$ in \eqref{eq:unitadj} is an order embedding. We do not follow their distinction, and speak of free objects in all cases.
\end{remark} 
If $H$ is any $\ell$-group, set
\begin{align}\nonumber
\Spec{H}\coloneqq \{\, \p \sst H\mid \p \text{ is a prime convex $\ell$-subgroup of } H \,\}.
\end{align}
Here, a {\em convex $\ell$-subgroup} $\ksg$ is a sublattice subgroup of $H$ that is order convex, and an {\em ideal} is a normal convex $\ell$-subgroup. For any convex $\ell$\nbd{-}subgroup~$\ksg$, the set of cosets $H/\ksg$ can be lattice ordered by: $\ksg x \leq \ksg y$ if, and only if, $x \leq ty$ in $H$ for some $t \in \ksg$. A convex $\ell$-subgroup $\p$ is \emph{prime} just when it is proper (i.e., $\p\neq H$), and the quotient lattice $H/\p$ is ordered. Throughout, we write `prime subgroup' to mean `prime convex $\ell$-subgroup', as usual. We denote by $\Specn{H}$ the subset of $\Spec{H}$ consisting of prime ideals.

We define a map 
\begin{equation}\label{eq:k1}
\kappa \colon \Pord{G}{\V} \longrightarrow \Spec \Free{G}{\V}
\end{equation}
as follows. Given a right pre-order $C \in \Pord{G}{\V}$, write 
\begin{equation}\nonumber
h_C \colon \Free{G}{\V} \longrightarrow H_C 
\end{equation}
for the unique $\ell$-homomorphism such that $h_C \circ \eta = R_C$. Write $H_C[\e]$ for the {\em stabiliser of $[\e]$}, that is, the set of $f \in H_C$ such that $f([\e]) = [\e]$. It is known that~$H_C[\e]$ is a prime subgroup of $H_C$ (see, e.g.,~\cite[Section 1.5]{Glass81}), and therefore $\iv{h_C}(H_C[\e])$ is a prime subgroup of $\Free{G}{\V}$. Hence, we set 
\begin{equation}\label{eq:k2}\nonumber
\kappa(C) \coloneqq \iv{h_C}(H_C[\e]) \in \Spec{\Free{G}{\V}},
\end{equation}
to complete the definition of $\kappa$ in \eqref{eq:k1}.

We define a map 
\begin{equation}\label{eq:p1}
\pi \colon \Spec \Free{G}{\V} \longrightarrow \Pord{G}{\V} 
\end{equation}
as follows. Given $\p \in \Spec \Free{G}{\V}$, we define the relation $\preceq_{\p}$ on $G$ by: 
\begin{equation}\label{eq:cp}
a \preceq_{\p} b\ \ \text{ if, and only if, }\ \ \p\eta(a) \leq \p\eta(b).
\end{equation} 
We shall prove in Lemma~\ref{l:welldefpi} below that $\preceq_{\p} \in \Pord{G}{\V}$. We write $C_\p$ for the positive cone of $\preceq_\p$. Finally, set
\begin{equation}\label{eq:p2}\nonumber
\pi(\p) := C_\p,
\end{equation}
to complete the definition of $\pi$ in \eqref{eq:p1}.

We topologise $\Spec{H}$ using the {\em spectral} (or {\em Zariski}) topology whose open sets are those of the form
\[
\Supp{(A)} = \{\, \p \in \Spec{H} \mid A \not\sst \p \,\},
\]
as $A$ ranges over arbitrary subsets of $H$~\cite[Proposition 49.6]{Darnel95}. Thus, the closed sets are those of the form 
\[
\Van{(A)} = \{\, \p \in \Spec{H} \mid A \sst \p \,\}.
\]
We also topologise $\Specn{H}$ by the subspace topology, with opens $\Supp^*{(A)}$ for $A \sst H$.
\noindent Further, we set
\begin{equation}\nonumber
\Pa{(a)} \coloneqq \{\, C \in \Pord{G}{\V} \mid a \in C \text{ and }\iv{a} \not \in C \,\}, \quad \text{ for }a \in G.
\end{equation} 
We endow $\Pord{G}{\V}$ with the smallest topology containing all sets $\Pa{(a)}$, and $\Bord{G}{\V}$ with the subspace topology. 

\begin{theorem}\label{t:main}
For any partially ordered group $G$ and any variety $\V$ of $\ell$-groups, the maps $\kappa \colon \Pord{G}{\V} \to \Spec{\Free{G}{\V}}$ and $\pi \colon \Spec{\Free{G}{\V}} \to \Pord{G}{\V}$ in \eqref{eq:k1} and \eqref{eq:p1} are mutually inverse, inclusion-preserving homeomorphisms that restrict to homeomorphisms between $\Bord{G}{\V}$ and $\Specn{\Free{G}{\V}}$.
\end{theorem}
\begin{remark}Definition~\ref{def:vpre} associates a class of right pre-orders on groups to any given variety $\V$ of $\ell$\nbd-groups, namely, $\Pord{G}{\V}$ as $G$ ranges over all groups; Theorem~\ref{t:main} establishes a non-trivial property of this association. We do not address in this paper the question of how to obtain a syntactic characterisation of the class of right pre-orders associated in this manner to a variety $\V$. For a more precise formulation of this problem, cf.\  Question~\ref{q:synt} below.  The construction leading to the statement of Theorem~\ref{t:main} makes it clear that one can also invert the correspondence: a class of right pre-orders on (a class of) groups uniquely determines a variety $\V$ of  $\ell$-groups. Again, a deeper investigation of this inverse correspondence is left to further research.
\end{remark}

Let $X$ be a topological space. A closed set $\emptyset \neq Y \sst X$ is {\em irreducible} if it is not the union of two proper closed subsets of itself. If every such $Y$ is the closure of a unique point, the space $X$ is called {\em sober}. A {\em generalised spectral} space is a sober space whose compact open subsets form a base closed under finite intersections. A {\em spectral} space is a generalised spectral space that is also compact~\cite{Hochster69}. A generalised spectral space $X$ is {\em completely normal}~\cite[Chapitre 10]{BigardEtAl1977} if for any $x, y \in X$ in the
closure of a singleton $\{\, z \,\}$, either $x$ is in the closure of $\{\, y \,\}$, or $y$ is in the closure of $\{\, x \,\}$. The {\em specialisation order} of a topological space $X$ is the relation defined on $X$ by: $x \preceq y$ if, and only if, $y$ is in the closure of $\{\, x \,\}$. 

\begin{corollary}\label{c:top}
For any partially ordered group $G$ and any variety $\V$ of $\ell$-groups, the space $\Pord{G}{\V}$ is a completely normal generalised spectral space whose specialisation order coincides with the inclusion order.
\end{corollary}

\section{Order-isomorphism}\label{s:orderiso}

\begin{lemma}\label{l:unitproperties}
For any partially ordered group $G$ and any variety $\V$ of $\ell$-groups, the image $\eta[G]\sst \Free{G}{\V}$ of $G$ under $\eta$ generates $\Free{G}{\V}$ as a lattice.
\end{lemma}

\begin{proof}
Write $\widehat{G}$ for the $\ell$-subgroup of $\Free{G}{\V}$ generated by $\eta[G]$. Then the positive group homomorphism $G\to \widehat{G}$ that agrees with $\eta$ on $G$ enjoys the universal property of $\eta$ because any $\ell$-homomorphism out of $\widehat{G}$ is uniquely determined by its action on any generating set of $\widehat{G}$. It follows by a standard argument on the uniqueness of universal arrows that $\widehat{G}=\Free{G}{\V}$. Since $\eta$ is a group homomorphism, $\eta[G]$ is a subgroup of $G$, and therefore must generate $\Free{G}{\V}$ as a lattice---it is elementary that in any $\ell$-group the lattice is distributive and the group operation distributes over meets and joins.
\end{proof}

If $H$ is an $\ell$-group and $\p \in \Spec{H}$, then the map  
\begin{align}\label{eq:Rrrc}
H &\xrightarrow{ \ R_\p \ } \Aut{H/\p} \\
x &\xmapsto{ \hspace{0.5cm} } R_\p(x)\colon \p y \mapsto \p yx\nonumber
\end{align}
is an $\ell$-homomorphism~\cite[Proposition 29.1]{Darnel95}. Note that $H/\p$ is naturally a totally ordered group if, and only if, $\p \in \Specn{H}$, and $R_\p[H]$ is isomorphic as an $\ell$-group to $H/\p$.

If $\p \in \Spec{\Free{G}{\V}}$, we write $\Omega_\p$ for the chain $\Free{G}{\V}/\p$.

\begin{lemma}\label{l:welldefpi}
For any partially ordered group $G$ and any variety $\V$ of $\ell$-groups, the map $\pi$ is a well-defined function with values in $\Pord{G}{\V}$.
\end{lemma}

\begin{proof}
It is easy to see that $\pi$ is a function from $\Spec{\Free{G}{\V}}$ into $\RPord{(G)}$, and it remains to show that $H_{C_\p} \in \V$ for any $\p \in \Spec{\Free{G}{\V}}$.

\begin{claim}\label{cl:chain}
For any $\p \in \Spec{\Free{G}{\V}}$ and any element $x \in \Free{G}{\V}$, it holds that $\p x = \p \eta(a)$, for some $a \in G$.
\end{claim}

\begin{proof}
Indeed, by Lemma~\ref{l:unitproperties}, each $x \in \Free{G}{\V}$ is of the form $\bigwedge_I\bigvee_{J_i} \eta(a_{ij})$, for $a_{ij} \in G$ and $i \in I, j \in J_i$ finite index sets. Since the chain $\Omega_\p$ is a lattice quotient of $\Free{G}{\V}$, we have $\p x = \bigwedge_I\bigvee_{J_i} \p\eta(a_{ij})$. Therefore, $\p x = \p \eta(a_{ij})$, for some $a_{ij} \in G$, and some $i \in I, j \in J_i$.
\end{proof}

\begin{claim}
For any $\p \in \Spec{\Free{G}{\V}}$, the images $R_\p[\eta[G]]$ and $R_{C_\p}[G]$ are isomorphic as groups.
\end{claim}

\begin{proof}
Given Claim~\ref{cl:chain}, it now follows from the definition of $C_\p$ that the map 
\[
\Omega_{C_\p} \xrightarrow{ \ \tau \ } \Omega_\p 
\]
defined by $[a] \mapsto \p\eta(a)$ is an order isomorphism between $\Omega_\p$ and the quotient $\Omega_{C_\p}$ of $G_{C_\p}$. (Recall that the notation $\Omega_C$, here applied to the case $C = C_\p$, was defined in Section~\ref{s:main}.) Therefore, the $\ell$-groups $\Aut{\Omega_\p}$ and $\Aut{\Omega_{C_\p}}$ are isomorphic through the map
\begin{align}\label{eq:tau}
\Aut{\Omega_{C_\p}} &\xrightarrow{ \ \widehat{\tau} \ } \Aut{\Omega_\p} 
\end{align}
sending $f \in \Aut{\Omega_{C_\p}}$ to the order automorphism $\widehat{\tau}(f) \colon \p\eta(a) \mapsto \tau(f([a]))$. Finally, the $\ell$\nbd{-}isomorphism $\widehat{\tau}$ restricts to a bijection between $R_\p[\eta[G]]$ and $R_{C_\p}[G]$. Indeed, for $b \in G$,
\[
\widehat{\tau}(R_{C_\p}(b)) \colon \p\eta(a) \mapsto \p\eta(ab),
\] 
that is, $\widehat{\tau}(R_{C_\p}(b)) = R_\p(\eta(b))$.
\end{proof}

By the preceding claim, and by the facts that the $\ell$-group $R_\p[\Free{G}{\V}]$ is generated by $R_\p[\eta[G]]$, and similarly $H_{C_\p}$ is generated by $R_{C_\p}[G]$, we infer that $R_\p[\Free{G}{\V}]$ and $H_{C_\p}$ are isomorphic. The homomorphic image $R_\p[\Free{G}{\V}]$ is in $\V$ because $\Free{G}{\V}$ is. Hence, $H_{C_\p} \in \V$, and this concludes the proof.
\end{proof} 

\begin{lemma}\label{l:inverse}
For any object $G$ in $\P$ and any variety $\V$ of $\ell$-groups, the maps $\kappa$ and $\pi$ are mutually inverse.
\end{lemma}

\begin{proof}
Let $C \in \Pord{G}{\V}$, and let $\p \coloneqq \iv{h_C}(H_C[\e])$. We show that $\pi \circ \kappa$ is the identity on $\Pord{G}{\V}$, that is $C_\p = C$. (Recall from \eqref{eq:cp} the definition of the pre-order associated to $C_\p$.)  

If $a \in C$, then 
\[
h_C(\eta(a) \wedge \e)([\e]) = R_C(a)([\e]) \wedge [\e] = [\e].
\]
Therefore, $(\eta(a) \wedge \e) \in \p$, and hence, $\p\eta(a) \geq \p \e$. This shows $C \sst C_\p$. Conversely, pick $a \in C_\p$, that is, $a$ is  such that $\p \e \leq \p\eta(a)$ in $\Omega_{\p}$. This means that $\e \leq t\eta(a)$, for some $t \in \p$. Hence, 
\[
h_C(t\eta(a) \wedge \e) = h_C(\e).
\]
Therefore, the element $h_C(t\eta(a) \wedge \e)$ is in the stabiliser of $[\e]$, which entails:
\begin{equation}\label{eq:eq1}
h_C(t\eta(a) \wedge \e)([\e]) = (h_C(t)h_C(\eta(a)) \wedge h_C(\e))([\e]) = [\e].
\end{equation}
Since $h_C \circ \eta = R_C$, from \eqref{eq:eq1} we obtain
\begin{equation}\label{eq:eq2}
R_C(a)(h_C(t)([\e])) \wedge h_C(\e)([\e]) =  [\e].
\end{equation}
But $t \in \p$, and thus $h_C(t)([\e]) = [\e]$; so, from \eqref{eq:eq2} we infer $R_C(a)([\e]) \wedge [\e] = [\e]$, i.e., $a \in C$.

To show that $\kappa \circ \pi$ is the identity on $\Spec{\Free{G}{\V}}$, we prove $\kappa(C_\p) = \p$ for a prime $\p$ of $\Free{G}{\V}$. By definition, $x \in \kappa(C_\p)$ if, and only if, $h_{C_\p}(x)([\e]) = [\e]$. By applying the map $\widehat{\tau}$ defined in \eqref{eq:tau}, this is equivalent to $R_\p(x)(\p \e) = \p \e$, that is, $x \in \p$.
\end{proof}

In order to show that $\kappa$ is order preserving, we begin by recording an easy observation.

\begin{proposition}\label{p:posnormform}
Let $H$ be an $\ell$-group generated by a subgroup $S\sst H$, and let $x\in H^+$. Then $x$ lies in the sublattice of $H$ generated by $\{\, s\vee e \mid s \in S \,\}$.
\end{proposition}

\begin{proof}
There are finite index sets $I$ and $J_i$ and elements $s_{ij}\in S$, $i\in I$ and $j\in J_i$, such that $x=\bigwedge_I\bigvee_{J_i} s_{ij}$. Since $x\geq \e$ we have $x\vee \e =x$, so we obtain $x=\left(\bigwedge_I\bigvee_{J_i} s_{ij}\right)\vee \e$. By distributivity, $x=\bigwedge_I\left(\bigvee_{J_i} s_{ij}\vee \e\right)$. In any lattice, 
\[
(a\vee b)\vee \e=(a\vee \e)\vee (b\vee \e),
\] 
so $x=\bigwedge_I\bigvee_{J_i} (s_{ij}\vee \e)$. 
\end{proof} 

\begin{lemma}\label{l:orderpres}
For any object $G$ in $\P$ and any variety $\V$ of $\ell$-groups, the map $\kappa$ is inclusion preserving.
\end{lemma}

\begin{proof}
Let $C \sst D \in \Pord{G}{\V}$, and pick $x = \bigwedge_{I}\bigvee_{J_i} (\eta(a_{ij}) \vee \e) \in \Free{G}{\V}^+$ such that $x \in \kappa(C)$, i.e., $h_C(x)([\e]) = [\e]$. This means:
\begin{align}\nonumber
h_C(x) & =  h_C\left(\bigwedge_{I}\bigvee_{J_i} (\eta(a_{ij}) \vee \eta(\e))\right) \\ \nonumber
& = \bigwedge_{I}\bigvee_{J_i} h_C(\eta(a_{ij}) \vee \eta(\e)) \\ \nonumber
& = \bigwedge_{I}\bigvee_{J_i} (R_C(a_{ij}) \vee R_C(\e)).
\end{align}
Hence, $h_C(x)([\e]) = [\e]$ if, and only if, 
\[
\bigwedge_{I}\bigvee_{J_i} ([a_{ij}] \vee [\e]) = [\e]  \text{ in } \Omega_C.
\]
Observe that $\bigvee_{J_i} ([a_{ij}] \vee [\e]) \geq_C [\e]$ for every $i \in I$ and hence, $h_C(x)([\e]) = [\e]$ if, and only if, 
\[
\bigvee_{J_{i^*}} ([a_{i^*j}] \vee [\e]) = [\e]
\] 
for some $i^* \in I$. Writing $J_{i^*} = \{\, 1, \dots, n \,\}$, and reindexing if necessary, we have 
\[
[a_{i^*1}] \leq_C [a_{i^*2}] \leq_C \dots \leq_C [a_{i^*n}] \leq_C [\e] \text{ in }\Omega_C,
\]
and hence, 
\[
[a_{i^*1}] \leq_D [a_{i^*2}] \leq_D \dots \leq_D [a_{i^*n}] \leq_D [\e]  \text{ in } \Omega_D.
\]
Therefore, 
\[
\bigwedge_{I}\bigvee_{J_i} ([a_{ij}] \vee [\e]) = [\e]  \text{ in } \Omega_D,
\]
which is equivalent to $h_D(x)([\e]) = [\e]$.
\end{proof}

\begin{theorem}\label{t:mutinv}
For any partially ordered group $G$ and any variety $\V$ of $\ell$-groups, the maps $\kappa$ and $\pi$ are mutually inverse, inclusion-preserving bijections.
\end{theorem}

\begin{proof}
By Lemma~\ref{l:inverse}, the maps $\kappa$ and $\pi$ are mutually inverse. By Lemma~\ref{l:orderpres}, $\kappa$ is inclusion preserving. If $\p \sst \q \in \Spec{\Free{G}{\V}}$, an element $a \in C_\p$ if, and only if, $\p \e \leq \p\eta(a)$. The latter is equivalent to $\e \leq t\eta(a)$, for some $t \in \p$. Hence, $\e \leq t\eta(a)$, for some $t \in \p \sst \q$, and therefore, $a \in C_\q$. Thus, $\pi$ is inclusion preserving.
\end{proof} 

\begin{lemma}\label{l:preord}
For any partially ordered group $G$ and any right pre-order $C$ on $G$, the quotient $\Omega_C$ is a totally ordered group with group operation $[a][b] = [ab]$ if, and only if, $C \in \BOrd{(G)}$. In that case, $H_C$  is isomorphic to $\Omega_C$. 
\end{lemma}

\begin{proof}
For a pre-order $C \in \BOrd{(G)}$, it is immediate that $\equiv_C$ is a group congruence, and $\Omega_C$ is a totally ordered group. Conversely, if $\Omega_C$ is a group with operation $[a][b] = [ab]$ totally ordered by $\leq_C$, we have that $a \preceq_C b$ implies $[sat] \leq_C [sbt]$ for every $a, b, t, s \in G$. That is, $sat \preceq_C sbt$.

Now, the map 
\begin{align}\nonumber
\Omega_C &\xrightarrow{ \ q \ } H_C 
\end{align}
defined by $[a] \mapsto \rho_C(a)$ is a group homomorphism. Moreover, $[a] <_C [b]$ if, and only if, $[ta] <_C [tb]$ for every $t \in G$. Hence, $q$ is an order isomorphism onto $\rho_C[G]$, and since the $\ell$-group $H_C$ generated by the totally ordered group~$\rho_C[G]$ is $\rho_C[G]$, the proof is complete. 
\end{proof}

Note that $q([a]) \in H_C[\e]$ if, and only if, $\rho_C(a)([\e]) = [\e]$, that is, $[a] = [\e]$.

\begin{theorem}\label{t:preordersn}
For any partially ordered group $G$ and any variety $\V$ of $\ell$-groups, if $C \in \Bord{G}{\V}$, then $\kappa(C)$ is a prime ideal of $\Free{G}{\V}$. Further, if $\p \in \Specn{\Free{G}{\V}}$, then $\pi(\p)$ is a pre-order on $G$.
\end{theorem}

\begin{proof}
For $C \in \Bord{G}{\V}$, suppose $x \in \kappa(C)$. We show $\iv{y}xy \in \kappa(C)$, for every $y \in \Free{G}{\V}$. By Lemma~\ref{l:preord}, we identify $H_C$ with $\Omega_C$, and have $h_C(x) = [\e]$. Similarly, given $y \in \Free{G}{\V}$, we have $h_C(y)=[b]$ for some $b \in G$. Therefore, 
\begin{equation}\nonumber
h_C(\iv{y}xy) = h_C(\iv{y})h_C(x)h_C(y)=  [\iv{b}][\e][b] = [\e].
\end{equation} 
If $\p \in \Spec{\Free{G}{\V}}$, and $a, b \in G$, we have $a \preceq_{\pi(\p)} b$ if, and only if, $\eta(a)\eta(\iv{b}) \leq x$ for some $x \in \p$. Therefore, if $\p$ is normal, we also have
\[
\eta(s)\eta(a)\eta(t)\eta(\iv{t})\eta(\iv{b})\eta(\iv{s}) \leq \eta(s)x\eta(\iv{s}) \in \p,
\]
which implies $sat \preceq_{\pi(\p)} sbt$, for every $s, t \in G$.
\end{proof}

\section{Homeomorphism}\label{s:homeo}

If $H$ is an $\ell$-group and $x \in H$, the {\em absolute value} $|x| \in H^+$ of $x$ is defined as $x \vee \iv{x}$. It is classical that the set $\Con{H}$ of convex $\ell$-subgroups of $H$ ordered by inclusion is a complete distributive sublattice of the lattice of subgroups of $H$~\cite[Propositions 7.5 and 7.10]{Darnel95}. Thus, in $\Con{H}$, meet is intersection and join $\bigvee \ksg_i$ is the subgroup of $H$ generated by $\bigcup \ksg_i$. We write~$\Ksg(S)$ to denote the convex $\ell$-subgroup generated by $S \sst H$. If $x \in H$, we write $\Ksg(x)$ for $\Ksg(\{\, x \,\})$, and call it the {\em principal convex $\ell$-subgroup} generated by $x$. 

\begin{proposition}\label{p:princcon}
For any $\ell$-group $H$, and for any $x, y \in H^+$, $z \in H$, the following hold.
\begin{enumerate}
\item [{\rm (1)}] $\Ksg(z) = \Ksg(|z|) = \{\, h \in H \mid |h| \leq |z|^n, \text{ for some }n \in \N\setminus\{\, 0 \,\} \,\}$. 
\item [{\rm (2)}] $\Ksg(x \wedge y) = \Ksg(x) \wedge \Ksg(y) = \Ksg(x) \cap \Ksg(y)$ and $\Ksg(x \vee y) = \Ksg(x) \vee \Ksg(y)$. 
\end{enumerate}
\end{proposition}

\begin{proof}
(1)~\cite[Proposition 7.13]{Darnel95}.\ (2)~\cite[Proposition 7.15]{Darnel95}.
\end{proof}

Throughout, we write $\Supp{(x)}$ in place of $\Supp{(\{\, x \,\})}$ for $x \in H$, and similarly for $\Van{(\{\, x \,\})}$.

\begin{proposition}\label{p:openbase}
For any $\ell$-group $H$, the set $\{\, \Supp{(x)} \mid x \in H \,\}$ is a base for the topology of $\Spec{H}$.
\end{proposition}

\begin{proof}
\cite[Proposition~49.7]{Darnel95}.
\end{proof}

\begin{proposition}\label{p:princsupp}
For any $\ell$-group $H$, and for any $x, y \in H^+$, $z \in H$, the following hold.
\begin{enumerate}
\item  [{\rm (1)}]  $\Supp{(z)} = \Supp{(|z|)}$. 
\item  [{\rm (2)}]  $\Supp{(x \wedge y)} = \Supp{(x)} \cap \Supp{(y)}$ and $\Supp{(x \vee y)} = \Supp{(x)} \cup \Supp{(y)}$. 
\end{enumerate}
\end{proposition}

\begin{proof}
(1) and (2) are immediate consequences of Proposition~\ref{p:princcon}. 
\end{proof}

\begin{theorem}\label{t:homeom}
For any partially ordered group $G$ and any variety $\V$ of $\ell$-groups, the maps $\kappa$ and $\pi$ are homeomorphisms.
\end{theorem}

\begin{proof}
Since $\kappa$ and $\pi$ are mutually inverse bijections by Theorem~\ref{t:mutinv}, it suffices to show that they both are open maps.

We first show that 
\begin{equation}\label{eq:kpa}
\kappa[\Pa{(a)}] = \Supp{(\eta(a)\vee \e)}, \quad \text{ for }a \in G.
\end{equation} 
Let $C \in \Pa{(a)}$. This means $R_C{(a)}([\e]) >_C [\e]$ in $\Omega_C$, that is, $h_C(\eta(a))([\e]) >_C [\e]$. Therefore, 
\[
h_C(\eta(a) \vee \e)([\e]) = [a] \vee [\e] >_C [\e],
\] 
and hence, $\iv{h_C}(H_C[\e]) \in \Supp(\eta(a) \vee \e)$. Similarly, for $\iv{h_C}(H_C[e]) \in \Supp(\eta(a) \vee \e)$, we prove $C \in \Pa{(a)}$. The assumption entails $h_C(\eta(a) \vee \e)([\e]) = [a] \vee [\e] >_C [\e]$. Since $\Omega_C$ is a chain, this can only happen if $[a] >_C [\e]$. Therefore, $a \in C$ and $\iv{a} \not \in C$. 

Since $\{\, \Pa{(a)} \mid a \in G \,\}$ is a subbase, and $\kappa$, being a bijection, preserves arbitrary intersections and unions, it follows that $\kappa$ is open.

To show $\pi$ is open, by Propositions~\ref{p:openbase} and~\ref{p:princsupp}.(1), together with the fact that $\pi$ is a bijection, it suffices to prove $\pi[\Supp{(x)}]$ is open, for $x \in \Free{G}{\V}^+$. By Proposition~\ref{p:posnormform}, 
\[
\Supp{(x)} = \Supp{(\bigwedge_I \bigvee_{J_i} (\eta(a_{ij}) \vee \e))}
\] 
for some finite index sets $I$ and $J_i$, and elements $a_{ij} \in G$. By the second item of Proposition~\ref{p:princsupp}, 
\[
\Supp{(x)} = \bigcap_I\bigcup_{J_i} \Supp(\eta(a_{ij}) \vee \e).
\] 
Since $\pi$ is a bijection,
\[
\pi[\Supp{(x)}] = \bigcap_I\bigcup_{J_i} \pi[\Supp{(\eta(a_{ij}) \vee \e)}].
\]
By \eqref{eq:kpa}, $\Pa{(a_{ij})} = \pi[\Supp{(\eta(a_{ij})\vee \e)}]$. Therefore, $\pi[\Supp{(x)}] = \bigcap_I\bigcup_{J_i} \Pa{(a_{ij})}$ is open.
\end{proof}

\begin{proof}[Proof of Theorem~\ref{t:main}]
Combine Theorems~\ref{t:mutinv},~\ref{t:preordersn}, and~\ref{t:homeom}.
\end{proof}

\section{Spectrality}\label{s:spec}

We write $\Conp{H}$ for the sublattice of $\Con{H}$ consisting of the principal convex $\ell$-subgroups of $H$; cf.\ Proposition~\ref{p:princcon}. Note that the lattice $\Conp{H}$ does have a minimum ($\Ksg(\e) = \{\, \e \,\}$), but not necessarily a maximum. We prove in this section that $\Spec{H}$ is the Stone dual of the distributive lattice $\Conp{H}$.

Let $D$ be a distributive lattice with a minimum, but not necessarily with a maximum. Topologise the set $X(D)$ of prime ideals of $D$ by declaring that the sets  
\begin{equation}\label{eq:widehat}
\widehat{a} = \{\, \jd \in X(D) \mid a \not\in \jd \,\}, \quad \text{for }a \in D
\end{equation}
form a subbase. The generated topology is known as the {\em Stone topology} of $X(D)$~\cite{Stone38}. 

Recall that an element $x$ of a lattice $D$ is {\em compact} if whenever $S \sst D$ is such that $\bigvee S$ exists, $x \leq \bigvee S$ implies $x \leq \bigvee T$, for some finite $T \sst S$.

\begin{proposition}\label{p:convcomp}
For any $\ell$-group $H$, the set $\Conp{H}$ consists precisely of the compact elements of $\Con{H}$. 
\end{proposition}

\begin{proof}
\cite[Proposition 7.16]{Darnel95}.
\end{proof}

\begin{theorem}\label{t:specdual}
For any $\ell$-group $H$, set $D \coloneqq \Conp{H}$. Then, the map
\begin{align}\label{eq:map}
X(D) & \xrightarrow{ \ h \ } \Con{H} \\\nonumber
\jd & \xmapsto{ \hskip .35 cm}  \bigvee\{\, \Ksg(x) \mid \Ksg(x) \in \jd \,\}
\end{align}
restricts to a homeomorphism between $X(D)$ and $\Spec{H}$. The compact open sets of $\Spec{H}$ are precisely those of the form $\Supp{(x)}$, for $x \in H$.
\end{theorem}

\begin{proof}

We first show $h$ defined in \eqref{eq:map} is a bijection onto the set $\Spec{H}$ of prime subgroups of $H$. 

Consider $\jd \in X(D)$. If $x \in h(\jd)$, then $\Ksg(x) \in \jd$. Indeed, by Proposition~\ref{p:convcomp}, there are finitely many $\Ksg(x_1), \dots, \Ksg(x_n) \in \jd$ such that $\Ksg(x) \subseteq \Ksg(x_1) \vee \dots \vee \Ksg(x_n)$. As $\jd$ is closed under finite joins and downward closed, we conclude $\Ksg(x) \in \jd$. Injectivity of $h$ is now obvious.

To prove primeness and surjectivity, we make recurrent use of Proposition~\ref{p:princcon}.(2). For primeness, if $x \wedge y \in h(\jd)$, then $\Ksg(x \wedge y) = \Ksg(x) \cap \Ksg(y) \in \jd$. Since $\jd$ is prime, either $\Ksg(x) \in \jd$ or $\Ksg(y) \in \jd$, that is, either $x \in h(\jd)$ or $y \in h(\jd)$. 

For surjectivity, we pick a prime subgroup $\p$ of $H$ and consider the set $\jd_\p = \{\, \Ksg(x) \mid x \in \p \,\}$. Clearly, $\jd_\p$ is downward closed and closed under finite joins. Now, $\Ksg(x) \cap \Ksg(y) \in \jd_\p$ is equivalent to $\Ksg(x \wedge y) \in \jd_\p$,  and the latter is equivalent to $x \wedge y \in \p$. Since $\p$ is prime, either $x \in \p$ or $y \in \p$, and hence, either $\Ksg(x) \in \jd_\p$ or $\Ksg(y) \in \jd_\p$. This shows that $\jd_\p$ is a prime ideal of $\Conp{H}$. Since, evidently, $\p = \bigvee\{\, \Ksg(x) \mid x \in \p \,\}$, we have $h(\jd_\p) = \p$.

Regarding now $h$ as a bijection $h \colon X(D) \to \Spec{H}$, we show that $h$ is a homeomorphism. Indeed, since for $x \in H$, $\Ksg(x) \in \jd$ if, and only if, $x \in h(\jd)$, we have
\begin{equation}\label{eq:openbij}
h[\widehat{\Ksg(x)}] = \{\, \p \in \Spec{H} \mid x \not\in \p \,\} = \Supp{(x)}.
\end{equation}
Since $h$ preserves arbitrary unions and intersections, this shows that $h$ is an open bijection. By Proposition~\ref{p:openbase}, it also shows that $h$ is continuous, and hence a homeomorphism. Finally, it is classical that the compact open sets of $X(D)$ are precisely those of the form $\widehat{\Ksg(x)}$ (see, e.g.,\ \cite{Johnstone86}), so that \eqref{eq:openbij} establishes the last assertion of the statement.
\end{proof} 
\begin{remark}[Reticulation of lattice-ordered groups] Theorem~\ref{t:specdual} exhibits the prime spectrum of an $\ell$-group $H$ via a purely lattice-theoretic construction, as the Stone dual of the lattice  $\Conp{H}$ of principal convex $\ell$\nbd{-}subgroups of $H$. In several variants, the result has circulated as folklore amongst researchers in the field. We have provided a full proof because we are not aware of a reference at this  level of generality. For Abelian lattice-groups with a strong unit, {\it alias} MV-algebras, see~\cite[and references therein]{GGM}. The construction in Theorem~\ref{t:specdual} is the exact analogue for lattice-groups of Simmons' well-known reticulation of a ring~\cite{Simmons1980}.
\end{remark}
An element $u$ of an $\ell$-group $H$ is a {\em strong (order) unit} if for all $x \in H$ there is $n \in \N \setminus \{\, 0 \,\}$ such that $x \leq u^n$; equivalently, by Proposition~\ref{p:princcon}.(1), if $\Ksg{(u)} = H$. 

\begin{corollary}\label{c:specspec}
For any $\ell$-group $H$, the space $\Spec{H}$ is a generalised spectral space. It is spectral if, and only if, the $\ell$-group $H$ has a strong unit.
\end{corollary}

\begin{proof}
It is a classical result that the Stone dual space of a distributive lattice $D$ with minimum is a generalised spectral space which is compact if, and only if, $D$ has a maximum (see~\cite{Stone38}; cf.\ \cite[Section 2.5]{Graetzer2011},~\cite[Section II.3]{Johnstone86}). Suppose now that $u \in H$ is a strong unit. Then, we have $\Ksg{(u)} = H$. Therefore, the lattice $\Conp{H}$ has a maximum, and its dual space $\Spec{H}$ is compact. Conversely, if $\Spec{H}$ is compact, then $\Spec{H} = \Supp{(u)}$ for some $u \in H$. But then, by the definition of $\Supp{(u)}$, every prime subgroup of $H$ excludes $u$. A standard argument (using Zorn's Lemma) then shows that every proper convex $\ell$-subgroup of $H$ omits $u$. Hence, $u$ is a strong unit.
\end{proof}

A poset is a {\em root system} if the upper bounds of any one of its elements form a chain. 

\begin{proposition}\label{p:rootsys}
For any $\ell$-group $H$ the poset $\Spec{H}$ is a root system, and the specialisation order of the generalised spectral space $\Spec{H}$ coincides with inclusion order.
\end{proposition}

\begin{proof}
It is standard that $\Spec{H}$ is a root system~\cite[Theorem 9.8]{Darnel95}. For the second statement, first note that, for any $\p \in \Spec{H}$, $\p \in \Van{(\p)}$ and every closed set $\Van{(A)}$ that contains $\p$ also contains $\Van{(\p)}$. Therefore, $\Van{(\p)}$ is the closure of $\p$. Further, for $\q \in \Spec{H}$, if $\p \sst \q$, then $\q \in \Van{(\p)}$, that is, $\p \leq \q$ in the specialisation order. Conversely, if the latter holds, then $\q \in \Van{(\p)}$, so that $\p \sst \q$.
\end{proof}

\begin{proof}[Proof of Corollary~\ref{c:top}]
Combine Theorem~\ref{t:homeom}, Corollary~\ref{c:specspec}, and Proposition~\ref{p:rootsys}.
\end{proof}

\begin{remark}Observe that
Corollary~\ref{c:specspec} provides a necessary and sufficient condition for $\Pord{G}{\V}$ to be compact, and hence spectral, namely, the existence of a strong  unit in $\Free{G}{\V}$. This yields a useful sufficient condition for spectrality: \emph{If the partially ordered group $G$ is finitely generated (as a group), then $\Pord{G}{\V}$ is compact, and hence spectral.} Indeed, in this case $\Free{G}{\V}$ is finitely generated (as an $\ell$-group), because it is generated by the image under $\eta$ of a generating set for $G$. But it is well known that an $\ell$-group with finitely many generators $x_1,\dots,x_n$ has $\bigvee_{i=1}^n|x_i|$ as strong unit; hence $\Pord{G}{\V}$ is compact. \end{remark}

\section{Characterisations of right pre-orders for specific varieties}\label{s:var}

For any chain $\Omega$, an $\ell$-subgroup $H$ of $\Aut{\Omega}$ is {\em transitive} (on $\Omega$)---equivalently, $H$ {\em acts transitively on $\Omega$}---if for every $r, s \in \Omega$ there exists $f \in H$ such that $f(r)=s$.

\begin{lemma}\label{l:transitive}
For any partially ordered group $G$ and any right pre-order $C$ on $G$, the $\ell$-group $H_C$ is transitive on the chain $\Omega_C$.
\end{lemma}

\begin{proof}
For every $a, b \in G$, the equivalence class $[a]$ is sent to $[b]$ by the map $R_C(\iv{a}b)$ as defined in \eqref{eq:Rrr}.
\end{proof}

We call an $\ell$-group {\em representable} if it is a subdirect product of totally ordered groups, and {\em Abelian} if its underlying group is Abelian. The class of representable $\ell$-groups is a variety of which Abelian $\ell$-groups form a subvariety~\cite[Proposition 9.3.3]{KM94}. We write $\Rep$ for the category whose objects are representable $\ell$-groups, and $\A$ for the category of Abelian $\ell$\nbd{-}groups. 

\begin{proposition}\label{p:reptrans1}
Every transitive representable $\ell$-group of order-preserving permutations of some chain is totally ordered.
\end{proposition}

\begin{proof}
\cite[Theorem 9.3.5]{KM94}.
\end{proof}

\begin{theorem}\label{t:reptrans2}
For any partially ordered group $G$ and any right pre-order $C$ on $G$, the following are equivalent.
\begin{enumerate}

\item The right pre-order $C$ is in $\Pord{G}{\Rep}$.

\item The $\ell$-group $H_C$ is totally ordered.

\item For every $a \in G$, either $ba\iv{b} \in C$ for every $b \in G$, or $ba\iv{b} \in \iv{C}$ for every $b \in G$.

\end{enumerate}
\end{theorem}

\begin{proof}
For (1) $\Leftrightarrow$ (2), combine Lemma~\ref{l:transitive} and Proposition~\ref{p:reptrans1}.

For (2) $\Leftrightarrow$ (3), note that $H_C$ is a chain if, and only if, for every $a \in G$, either $[t] \leq_C [ta]$ for every $t \in G$, or $[ta] \leq_C [t]$ for every $t \in G$. Equivalently, for every $a \in G$, either $t \preceq_C ta$ for every $t \in G$, or $ta \preceq_C t$ for every $t \in G$, that is, either $\e \preceq_C ta\iv{t}$ for every $t \in G$, or $ta\iv{t} \preceq_C \e$ for every $t \in G$. 
\end{proof}

\begin{theorem}\label{t:reptrans3}
For any partially ordered group $G$ and any right pre-order $C$ on~$G$, the following are equivalent.
\begin{enumerate}

\item The right pre-order $C$ is in $\Pord{G}{\A}$.

\item The $\ell$-group $H_C$ is totally ordered Abelian.

\item For every $a, b \in G$, $[\iv{a}\iv{b}ab] = [\e]$ in $\Omega_C$.

\end{enumerate}
\end{theorem}

\begin{proof}
For (1) $\Rightarrow$ (2), if $C$ is in $\Pord{G}{\A}$, the $\ell$-group $H_C$ is Abelian, and totally ordered by Lemma~\ref{l:transitive}.

For (2) $\Rightarrow$ (3), observe that $H_C$ is Abelian if, and only if, $R_C[G]$ is Abelian (cf.\ \cite[1.1]{Conrad1970}). Thus if, and only if, for every $a, b, t \in G$, $[tab] = [tba]$ in $\Omega_C$. The latter entails $[ab] = [ba]$ for every $a, b \in G$, that is, $[ab\iv{a}\iv{b}] = [\e]$,  for every $a, b \in G$.

For (3) $\Rightarrow$ (1), pick a right pre-order $C$ on $G$ satisfying $[\iv{a}\iv{b}ab] = [\e]$, for every $a, b \in G$. We show that $[tab] = [tba]$ in $\Omega_C$ for every $a, b, t \in G$. Indeed, since $[ab] = [ba]$ and $C$ is right invariant, also $[abt] =[bat]$, for every $t \in G$. By using the assumption (3) again, $[bat]=[tba]$ and $[abt] = [tab]$. Thus, $[tab]=[tba]$.
\end{proof}
\begin{remark}\label{rem:synt} The class of right pre-orders on groups is elementary---that is, axiomatisable in first-order logic---in the language of partially ordered groups. By Theorems~\ref{t:reptrans2} and~\ref{t:reptrans3}, so are the notions of `representable' and `Abelian' right pre-orders on groups.  (For a related, recent connection between lattice-groups and logic, see~\cite{ColacitoMetcalfe2019}.)
\end{remark}
\begin{problem}\label{q:synt}For which varieties of $\ell$-groups is the corresponding class of right pre-orders on groups, as provided by Definition ~\ref{def:vpre}, elementary in the language of partially ordered groups?
\end{problem}
For any partially ordered group $G$ and any variety $\V$ of $\ell$-groups, consider the factorisation of the universal map $\eta \colon G \to \Free{G}{\V}$ given by
\begin{align}\label{eq:factor}
G & \xrightarrow{ \ \zeta \ }\ \eta[G] \xrightarrow{ \ \xi \ } \Free{G}{\V},
\end{align}
where $\eta[G]$ is the group image of $G$ under $\eta$ partially ordered by the restriction of the order on $\Free{G}{\V}$, and $\xi$ is the inclusion map. 

\begin{proposition}\label{p:ontointo}
The group homomorphism $\xi \colon \eta[G] \to \Free{G}{\V}$ from \eqref{eq:factor} is an order embedding satisfying the universal property \eqref{eq:universal}.
\end{proposition}

\begin{proof}
It is evident by construction that $\xi$ is an order embedding. Consider the following commutative diagrams
\[
\begin{tikzcd}
G \arrow[rd, "\eta^* \circ \, \zeta"'] \arrow[r, "\eta"] & \Free{G}{\V} \arrow[d,  "h", "!"'] & \text{ and } & \eta[G] \arrow[rd, "\xi"'] \arrow[r, "\eta^*"] & \Free{\eta[G]}{\V} \arrow[d,  "k", "!"']\\
& \Free{\eta[G]}{\V} && & \Free{G}{\V}
\end{tikzcd}
\]
We prove that $h \circ k$ and $k \circ h$ are, respectively, the identity map on $\Free{\eta[G]}{\V}$ and the identity map on $\Free{G}{\V}$. Indeed, 
\[
h \circ k \circ \eta^* \circ \zeta = h \circ \xi \circ \zeta = h \circ \eta = \eta^* \circ \zeta,
\] 
and since $\zeta$ is an epimorphism, we get $h \circ k \circ \eta^* = \eta^*$. Similarly, 
\[
k \circ h \circ \eta = k \circ h \circ \xi \circ \zeta = k \circ \eta^* \circ \zeta = \xi \circ \zeta = \eta.
\]
By the universal property of $\eta$ and $\eta^*$, we infer the thesis.
\end{proof}

\begin{remark}\label{r:pordspec}
Proposition~\ref{p:ontointo} entails that the space $\Pord{\eta[G]}{\V}$ is homeomorphic to  $\Spec{\Free{G}{\V}}$ and hence, by Theorem~\ref{t:main}, to the space $\Pord{G}{\V}$.
\end{remark}

For a partially ordered group $G$, a {\em right order} on $G$ is a right invariant proper total order on $G$ extending $G^+$, and an {\em order} on $G$ is a right order on~$G$ which is also left invariant. We call {\em (right) orderable} a partially ordered group $G$ that can be equipped with a (right) order. We say that a partially ordered group $G$ is {\em isolated} if $a^n \in G^+$ for some $n \in \N \setminus \{\, 0 \,\}$ implies $a \in G^+$, for every $a \in G$. As usual, a group $G$ for which $a^n = \e$ for some $n \in \N \setminus \{\, 0 \,\}$ implies $a = \e$, for every $a \in G$, is called {\em torsion-free}. 

\begin{lemma}\label{l:intersection}
Suppose $\V$ is the variety of all $\ell$-groups (respectively, representable $\ell$-groups). For any partially ordered group $G$, the universal map $\eta \colon G \to \Free{G}{\V}$ is an order embedding if, and only if, the positive cone $G^+$ is the intersection of the right orders (respectively, the orders) on $G$; for the variety $\V$ of Abelian $\ell$-groups, $\eta \colon G \to \Free{G}{V}$ is an order embedding if, and only if, $G$ is an isolated partially ordered Abelian group.
\end{lemma}

\begin{proof}
For the variety $\L$ of all $\ell$-groups, see~\cite[Th\'eor\`eme A.2.2]{BigardEtAl1977}. For the variety $\Rep$ of representable $\ell$\nbd{-}groups, see~\cite[Note de l'appendice]{BigardEtAl1977}. For the variety $\A$ of Abelian $\ell$-groups, it suffices to observe that the free Abelian $\ell$-group $\Free{G}{\A}$ over $G$ is the free $\ell$-group $\Free{G}{}$ over $G$ if $G$ is Abelian~\cite[1.2]{Conrad1970}.
\end{proof}

\begin{remark}\label{r:isolated}
For a partially ordered Abelian group $G$, being isolated is equivalent to $G^+$ being intersection of the (right) orders that extend it~\cite[Corollaire A.2.6]{BigardEtAl1977}.
\end{remark}

\begin{corollary}\label{c:orderability}
Suppose $\V$ is the variety of all $\ell$-groups (respectively, representable $\ell$-groups). For any group $G$, the universal map $\eta \colon G \to \Free{G}{\V}$ is injective if, and only if, $G$ is right orderable (respectively, orderable); for the variety $\V$ of Abelian $\ell$-groups, $\eta \colon G \to \Free{G}{V}$ is injective if, and only if, $G$ is non-trivial torsion-free Abelian.
\end{corollary}

\begin{proof}
From Lemma~\ref{l:intersection} considering a group $G$ as the partially ordered group $G$ with the trivial order.
\end{proof} 

The isomorphism between the poset of right pre-orders on $G$ and $\RPord{(G)}$ restricts to a bijection between the set of right orders on $G$ and the subset~$\Rord{(G)}$ of pre-cones $C \in \RPord{(G)}$ such that $C \cap \iv{C} = \{\, \e \,\}$. We now topologise~$\Rord{(G)}$  with the subspace topology inherited from $\RPord{(G)}$, with subbase (of clopens) consisting of 
\begin{equation}\label{eq:rotop}
\{\, C \in \Rord{(G)} \mid a \in C \,\}\ \ \text{ and }\ \ \{\, C \in \Rord{(G)} \mid a \not\in C \,\},
\end{equation} 
as $a$ ranges in $G$.
Similarly, the set of orders on $G$ is in bijection with the subset $\Ord{(G)}$ of normal pre-cones $C \in \Rord{(G)}$, and we equip $\Ord{(G)}$ with the subspace topology.

\begin{remark}\label{r:incomp}
For any variety $\V$ of $\ell$-groups, if a right pre-order $C \in \Pord{G}{\V}$ is a right order then $C$ must be inclusion minimal in $\Pord{G}{\V}$. In fact, any proper subset $D \subset C$ would fail the condition $G = D \cup \iv{D}$. Conversely, suppose $C \in \Pord{G}{\V}$ is minimal. Remark~\ref{r:pordspec} provides a natural way to associate to $C$ a minimal element of $\Pord{\eta[G]}{\V}$. However, there is no {\it a priori} reason why the latter should be a right order. We shall presently see that this is the case for the varieties of all $\ell$-groups, of representable $\ell$-groups, and of Abelian $\ell$-groups. 
\end{remark}
\begin{problem}Characterise the varieties $\V$ of $\ell$-groups such that, for all partially ordered groups $G$, the minimal elements of $\Pord{G}{\V}$ are right orders on~$\eta[G]$.
\end{problem}

\begin{proposition}\label{p:representable}
For any $\ell$-group $H$, the following are equivalent.
\begin{enumerate}
\item $H$ is representable.
\item Each minimal prime subgroup is an ideal.
\end{enumerate}
\end{proposition}

\begin{proof}
See, e.g.,\ \cite[Proposition 47.1]{Darnel95}.
\end{proof}

\begin{theorem}\label{t:minorders}
Suppose $\V$ is the variety of all $\ell$-groups (respectively, representable  or Abelian $\ell$-groups). For any partially ordered group $G$, the minimal layer of $\Pord{G}{\V}$ is homeomorphic to the space of right orders (respectively, orders) on $\eta[G]$.
\end{theorem}

\begin{proof}
For $\eta \colon G \to \Free{G}{}$, the space of right orders on $\eta[G]$ is nonempty by Lemma~\ref{l:intersection}. By Remark~\ref{r:incomp}, the space of right orders on $\eta[G]$ is made of minimal elements of $\Pord{\eta[G]}{}$. We now show that every $C \in \RPord{(\eta[G])}$ extends a right order. Let $P$ be a right order on $\eta[G]$, and $P(C)$ be its restriction $P \cap (C \cap \iv{C})$. Consider the binary relation on $\eta[G]$ defined by: 
\begin{equation}\label{eq:minro}
a \leq b \quad \Longleftrightarrow \quad [a] < [b] \text{ or } ([a] = [b] \text{ and } e \leq_{P(C)} b\iv{a}), \text{ for }a,b \in G.
\end{equation}
Then, the relation $\leq$ is a right order on $\eta[G]$ that extends $\eta[G]^+$, and $a \leq b$ implies $a \preceq_C b$. It is elementary that $\leq$ is a total order. Suppose now that $a \leq b$ because $[a] < [b]$. But then, $a \prec_C b$, and hence, $at \prec_C bt$, which means $[at] < [bt]$. On the other hand, if $[a] = [b]$ and $\e \leq_{P(C)} b\iv{a}$, then $[ac] = [bc]$ and $\e \leq_{P(C)} bc\iv{c}\iv{a}$. Finally, it is clear that if $a \leq b$, then $a \preceq_C b$.

For $\eta \colon G \to \Free{G}{\Rep}$, Lemma~\ref{l:intersection} entails that the space of orders on $\eta[G]$ is nonempty, and by Remark~\ref{r:incomp}, the space of orders on $\eta[G]$ is made of minimal elements of $\Pord{\eta[G]}{\Rep}$. We pick an order $P$ on $\eta[G]$ and its restriction $P(C)$, and show that if $C \in \Bord{\eta[G]}{\Rep}$, the binary relation $\leq$ defined in \eqref{eq:minro} is an order on $\eta[G]$ included in $C$. We can then conclude using Proposition~\ref{p:representable} and Theorem~\ref{t:preordersn}, from which we obtain that every minimal element $C$ of~$\Pord{\eta[G]}{\Rep}$ is a pre-order. Hence, we only need to prove that the right order $\leq$ that we obtain is also left invariant. For this, suppose that $a \leq b$ because $[a] < [b]$. Then, this means that $a \prec_C b$, and hence, $\e \prec_C b\iv{a}$. Now, by Theorem~\ref{t:reptrans2}, $\e \preceq_C cb\iv{a}\iv{c}$. If we had also $cb\iv{a}\iv{c} \preceq_C \e$, we would get a contradiction with $b\iv{a} \npreccurlyeq_C \e$ since $C$ is a pre-order. Therefore, $ca \prec_C cb$, and hence, $[ca] < [cb]$. Assume now that $[a] = [b] \text{ and } \e \leq_{P(C)} b\iv{a}$. If $b\iv{a} \in P(C)$, also $cb\iv{a}\iv{c} \in P(C)$. The latter entails $[ca] = [cb]$ and $e \leq_{P(C)} cb\iv{a}\iv{c}$. Therefore, if $a \leq b$, also $ca \leq cb$, and the right order $\leq$ is in fact an order on $\eta[G]$.

For $\eta \colon G \to \Free{G}{\A}$, observe that $\eta[G]$ is an isolated partially ordered Abelian group and hence, from Remark~\ref{r:incomp}, $\eta[G]^+$ is the intersection of the orders that extend it. Moreover, the free $\ell$-group $\Free{\eta[G]}{}$ over $\eta[G]$ is the free Abelian $\ell$-group $\Free{\eta[G]}{\A}$ over $\eta[G]$~\cite[1.2]{Conrad1970}. Thus, $\Pord{\eta[G]}{\A}$ is $\RPord{(\eta[G])}$ and hence, the minimal layer of $\Pord{\eta[G]}{\A}$ is the space of (right) orders on~$\eta[G]$.

The results now follow from Remark~\ref{r:pordspec}.
\end{proof}

For an $\ell$-group $H$, we write $\Min{H}$ for the set of minimal prime subgroups of $H$, and we topologise it with the subspace topology from $\Spec{H}$. We write $\Supp_{\m}{(A)}$ (respectively, $\Van_{\m}{(A)}$) for open subsets (respectively, closed subsets) of $\Min{H}$ with $A$ ranging over arbitrary subsets of $H$. 

\begin{corollary}\label{c:minetag}
Suppose $\V$ is the variety of all $\ell$-groups (respectively, representable  or Abelian $\ell$-groups). For any partially ordered group $G$, $\Min{\Free{G}{\V}}$ is homeomorphic to the space of right orders (respectively, orders) on $\eta[G]$.
\end{corollary}

\begin{proof}
Combine Corollary~\ref{c:top} and Theorem~\ref{t:minorders}.
\end{proof}
\begin{remark}\label{r:minorders}We do not know at this stage whether a characterisation of the minimal elements of $\Pord{G}{\V}$ along the lines of Theorem~\ref{t:minorders} is feasible, even in the case of  well-studied varieties $\V$ of $\ell$-groups. For example, suppose $\V$ is the variety of normal-valued $\ell$-groups~\cite[Th\'eor\`eme 4.3.10]{BigardEtAl1977}. Suppose further that $G$ is a group admitting a Conradian right order~\cite[Theorem 9.5]{CR2016}. Then one can prove that each Conradian right order on $G$ is a minimal element of $\Pord{G}{\V}$. However, it is unclear to us at present whether each minimal member of $\Pord{G}{\V}$ is a Conradian right order on $G$. 
 
\end{remark} 
\section{Minimal spectra}\label{s:min}

For any $\ell$-group $H$, we adopt the standard notation $x\perp y$---read `$x$ and $y$ are orthogonal'---to denote $|x|\wedge |y|=\e$, for $x,y \in H$. For $S \sst H$, we set 
\[
S^{\perp} \coloneqq \{\, x \in H \mid x \perp y\text{ for all }y \in S \,\};
\] 
we write $S^{\perp\perp}$ instead of $(S^{\perp})^{\perp}$, and $x^{\perp}$ instead of $\{\, x \,\}^{\perp}$ for $x \in H$. It is clear that $x^{\perp} = |x|^{\perp}$ for every $x \in H$. A subset $T \sst H$ is a {\em polar} if it satisfies $T = T^{\perp\perp}$ or, equivalently, if there exists $S \sst H$ such that $T = S^{\perp}$. We write $\Pol{H}$ for the set of polars of $H$. Under the inclusion order, $\Pol{H}$ is a complete distributive lattice with $H=\e^{\perp}$ as its maximum, $\{\, \e \,\} = H^{\perp} = \e^{\perp\perp}$ as its minimum, meets given by intersection, and joins given by $\bigvee T_i = (\bigcup T_i)^{\perp\perp}$. It can be shown that $\Pol{H}$ is a complete Boolean algebra, with complementation given by the map $T \mapsto T^{\perp}$. If $x \in H$, the set $x^{\perp\perp}$ is called the {\em principal polar} generated by $x$. Then, $x^{\perp\perp}$ is the inclusion-smallest polar containing $x$. We write $\Polp{H}$ for the set of principal polars of $H$; it is a sublattice of $\Pol{H}$ because of the identitities 
\begin{equation}\label{eq:prinpol1}
(x \wedge y)^{\perp\perp} = x^{\perp\perp} \cap y^{\perp\perp},
\end{equation}
\begin{equation}\label{eq:prinpol2}
(x \vee y)^{\perp\perp} = x^{\perp\perp} \vee y^{\perp\perp},
\end{equation}
which hold for every $x, y \in H^+$. The minimum $\e^{\perp\perp}$ of $\Pol{H}$ lies in $\Polp{H}$, while the maximum $H=\e^{\perp}$ is principal if, and only if, $H$ has a {\em weak (order) unit}---an element $w \in H^+$ such that for each $x \in H$, $w \wedge |x| = \e$ implies $x = \e$. In that case, $w^{\perp\perp} = H$. Note that the existence of a weak unit is not sufficient for $\Polp{H}$ to be a Boolean subalgebra of $\Pol{H}$, because the complement of a principal polar need not be principal (cf.\ Theorem~\ref{t:bapolars}).  

We recall here the standard characterisation of minimal primes. 

\begin{proposition}\label{p:minprimes}
For any $\ell$-group $H$ and any $\p\in\Spec{H}$, the following are equivalent.
\begin{enumerate}
\item The prime $\p$ is minimal.
\item $\p = \bigcup \{\, x^{\perp} \mid x \not \in \p \,\}$.
\end{enumerate}
\end{proposition}

\begin{proof}
\cite[Th\'eor\`eme 3.4.13]{BigardEtAl1977}.
\end{proof}

\noindent From Proposition~\ref{p:minprimes}, an element $w \in H^+$ is a weak unit if, and only if, $w$ misses every minimal prime.

\begin{lemma}\label{l:latpolp}
The map 
\begin{equation}\label{eq:ontolh}
\Conp{H} \xrightarrow{ \ f \ } \Polp{H}
\end{equation}
defined by $\Ksg{(x)} \mapsto x^{\perp\perp}$ is an onto lattice homomorphism preserving minimum.
\end{lemma}

\begin{proof}
The map $f$ is well defined, since $S^{\perp} = \Ksg{(S)}^\perp$ for $S \sst H$~\cite[3.2.5]{BigardEtAl1977}, and clearly onto. Moreover, it is a lattice homomorphism by Proposition~\ref{p:princcon} and \eqref{eq:prinpol1}--\eqref{eq:prinpol2}, and preserves the minimum since $\{\, \e \,\} \mapsto \e^{\perp\perp}$.
\end{proof}

Recall the notation $X(D)$ for the Stone dual of a distributive lattice $D$ with minimum, and set for the rest of this section $D \coloneqq \Polp{H}$. In light of Theorem~\ref{t:specdual}, we identify the Stone dual of $\Conp{H}$ with $\Spec{H}$. By general considerations, the map $f$ in \eqref{eq:ontolh} has a Stone dual $f^*$ defined by
\begin{align}\label{eq:fstar}
X(D) &\xrightarrow{ \ f^* \ } \Spec{H} \\ \nonumber
{\jd} &\xmapsto{ \hspace{0.45cm} } \bigvee \{\,  \Ksg{(x)} \mid x^{\perp\perp} \in \jd \,\}.
\end{align}
 Our next aim is to characterise the range of $f^*$. To this end, we introduce a notion that appears to be new with the present paper:

\begin{definition}[Quasi-minimal prime subgroups]
For any $\ell$-group $H$, a prime subgroup $\p \in \Spec{H}$ is {\em quasi minimal} if $\p = \bigcup\{\, x^{\perp\perp} \mid x \in \p \,\}$. The {\em quasi-minimal spectrum} $\Qin{H}$ of $H$ is the subset of quasi-minimal prime subgroups equipped with the subspace topology inherited from $\Spec{H}$.
\end{definition}

We write $\{\, \Supp_{\q}{(x)} \,\}_{x \in H}$ for the open base  induced by $\{\, \Supp{(x)} \,\}_{x \in H}$ on $\Qin{H}$ by restriction. 

\begin{lemma}\label{l:queen}
For any $\ell$-group $H$, the image $f^{*}[X(D)]$ coincides with the quasi-minimal spectrum $\Qin{H}$. Further, $f^{*}$ is a homeomorphism onto its range.
\end{lemma}
\begin{proof}
We establish the following equivalent description of~$f^*(\jd)$, for any $\jd \in X(D)$:
\begin{equation}\label{eq:darnelmap}
f^*(\jd) = \{\, x \in H \mid x^{\perp\perp} \in \jd \,\}.
\end{equation}
(Compare with Darnel's construction in~\cite[Proposition 49.18]{Darnel95}.) First, observe that $y \in f^*(\jd)$ implies $\Ksg{(y)} \sst \Ksg{(x_1)} \vee \dots \vee \Ksg{(x_n)}$ for some $x_1^{\perp\perp}, \dots, x_n^{\perp\perp} \in \jd$ by Proposition~\ref{p:convcomp}. Further, from Proposition~\ref{p:princcon} it follows $\Ksg{(y)} \sst \Ksg{(x)}$ for some $x^{\perp\perp} \in \jd$, since $\jd$ is closed under finite joins. Thus, from $\Ksg{(y)} \sst \Ksg{(x)}$ we obtain $y^{\perp\perp} \sst x^{\perp\perp}$, which allows the conclusion $y \in  \{\, x \in H \mid x^{\perp\perp} \in \jd \,\}$. Conversely, $y^{\perp\perp} \in \jd$, then $\Ksg{(y)} \sst f^*(\jd)$ by \eqref{eq:fstar}, and hence, $y \in f^*(\jd)$. 

Now, if $y \in x^{\perp\perp}$ for $x \in f^{*}(\jd)$, then $y^{\perp\perp} \sst x^{\perp\perp}$, and hence $y^{\perp\perp} \in \jd$ by downward closure of $\jd$. Thus, $y \in f^{*}(\jd)$, which entails $f^{*}(\jd) \in \Qin{H}$. For the remaining inclusion, suppose $\p \in \Qin{H}$. We prove that 
\[
\jd_\p \coloneqq \{\, x^{\perp\perp} \mid x \in \p \,\}
\] 
is a prime ideal of $D$, and hence, is the pre-image of $\p$ under $f^*$. It is elementary that $\jd_\p$ is an ideal of $D$. Now, by \eqref{eq:prinpol1}, $x^{\perp\perp} \wedge y^{\perp\perp}=x^{\perp\perp} \cap y^{\perp\perp} \in \jd_\p$ is equivalent to $(x \wedge y)^{\perp\perp} \in \jd_\p$, that is, $x \wedge y \in \p$. By primeness of $\p$, either $x \in \p$ or $y \in \p$, from which the desired conclusion.

Note that injectivity of $f^{*}$ is now immediate from \eqref{eq:darnelmap}. Finally, to show that $f^{*}$ is a homeomorphism onto its range it suffices to observe that \eqref{eq:darnelmap} entails $f^{*}[\widehat{(x^{\perp\perp})}] = \Supp_{\q}{(x)}$, where $x$ ranges over $H$, and $\widehat{(x^{\perp\perp})}$ is the set of all prime ideals of $D$ not containing $x^{\perp\perp}$; cf.\ \eqref{eq:widehat}. 
\end{proof}
We record a consequence that provides for lattice-groups the spectral equivalent of the existence of a weak unit.
\begin{corollary}
For any $\ell$-group $H$, there exists a weak unit $w \in H$ if, and only if, $\Qin{H}$ is compact.
\end{corollary}

\begin{proof}
This is an immediate consequence of Lemma~\ref{l:queen} along with standard Stone duality (cf.\ Corollary~\ref{c:specspec}).
\end{proof}

\begin{theorem}\label{t:bapolars}
For any $\ell$-group $H$, we have $\Min{H} \sst \Qin{H}$, and the following are equivalent.
\begin{enumerate}
\item $\Polp{H}$ is a Boolean subalgebra of $\Pol{H}$.

\item $\Min{H}$ is compact. 

\item $H^+$ is {\em complemented}: for every $x \in H^+$ there is $y \in H^+$ such that $x \wedge y = \e$ and $x \vee y$ is a weak unit.
\end{enumerate}
If any one of the equivalent conditions (1)--(3) holds, then $\Min{H} = \Qin{H}$.
\end{theorem}

\begin{proof}
By Proposition~\ref{p:minprimes}, if $x \in \m$ and $\m \in \Min{H}$, there is an element $y \not\in \m$ such that $y \in x^\perp$. Moreover, every element $z \in x^{\perp\perp}$ is also an element of $y^\perp$, that is, $z \in \m$. Therefore, $\m \in \Qin{H}$. 

For (1) $\Rightarrow$ (2), observe that if $\Polp{H}$ is a Boolean algebra, then for every $x \in H$, we have $x^{\perp} = y^{\perp\perp}$ for some $y \in H$. Thus, a quasi-minimal prime $\p$ is 
\begin{equation}\nonumber
\p =\bigcup \{\, x^{\perp\perp} \mid x \in \p \,\} = \bigcup\{\, y^{\perp} \mid y^{\perp} = x^{\perp\perp}\text{ for some }x \in \p \,\}.
\end{equation} 
Now, if $y \in \p$ for some of those $y$, then $y^{\perp\perp} \sst \p$ and hence, 
\[
x^{\perp\perp} \vee y^{\perp\perp} = (x \vee y)^{\perp\perp} = H \sst \p,
\] 
which is a contradiction. Moreover, if $y \not\in \p$, then $y \not \in \m$ for every minimal prime $\m \sst \p$, that is, $y^{\perp} \sst \m$ for every minimal prime $\m \sst \p$. Hence, $y^{\perp} \sst \p$. Therefore, every quasi-minimal prime $\p \in \Qin{H}$ is in fact minimal. Since~$\Qin{H}$ is the Stone dual space of a Boolean algebra, it is compact, and hence so is $\Min{H}$.

For (2) $\Rightarrow$ (3), we first observe that 
\begin{equation}\nonumber
\Supp_{\m}{(x)} = \Van_{\m}{(x^{\perp})},
\end{equation}
and hence 
\begin{equation}\nonumber
\Van_{\m}{(x)} =  \bigcup_{z \in x^{\perp}} \Supp_{\m}{(z)} = \Supp_{\m}{(y_1 \vee \dots \vee y_n)},
\end{equation}
for some $y_1, \dots, y_n \in x^{\perp}$, where the last equality follows from $\Van_{\m}{(x)}$ being closed in a compact space, and from Proposition~\ref{p:princsupp}. Let $y \coloneqq y_1 \vee \dots \vee y_n$. We show that $|x| \wedge |y| = \e$, and $\Supp_{\m}(|x| \vee |y|) = \Min{H}$. In fact, 
\begin{equation}\nonumber
\Supp_{\m}(|x| \vee |y|) = \Supp_{\m}{(x)} \cup \Supp_{\m}{(y)} = \Supp_{\m}{(x)} \cup \Van_{\m}{(x)} = \Min{H},
\end{equation}
that is, $|x| \vee |y|$ is a weak unit. Further, 
\begin{equation}\nonumber
\Supp_{\m}(|x| \wedge |y|) = \Supp_{\m}{(x)} \cap \Supp_{\m}{(y)} = \Supp_{\m}{(x)} \cap \Van_{\m}{(x)} = \emptyset,
\end{equation}
or equivalently, $|x| \wedge |y| = \e$ since $\bigcap_{\m \in \Min{H}} \m = \{\, \e \,\}$. 

For (3) $\Rightarrow$ (1), observe that two positive elements $x, y \in H^{+}$ are orthogonal if, and only if, $x^{\perp\perp} \cap y^{\perp\perp} = \{\, \e \,\}$ by \eqref{eq:prinpol1}. Similarly by \eqref{eq:prinpol2}, $x \vee y$ is a weak unit if, and only if, $x^{\perp\perp} \vee y^{\perp\perp} = H$. Now, since $x^{\perp} = {|x|}^{\perp}$ for every $x \in H$, the proof is complete. 

Finally, we show that if $\Min{H}$ is compact, then $\Min{H} = \Qin{H}$. For this, assume $\p \in \Qin{H}\setminus\Min{H}$. Then for every $\m \in \Min{H}$, there is $x \in \p$ such that $x \not \in \m$. Hence, 
\[
\Min{H} \sst \bigcup \{\, \Supp{(x)} \mid x \in \p \,\}.
\]
If we assume compactness of $\Min{H}$, we have $\Min{H} \sst \Supp{(x_1 \vee \dots \vee x_n)}$ for some $w \coloneqq x_1 \vee \dots \vee x_n \in \p$. Hence, the prime $\p$ contains the weak unit $w$, which is a contradiction since $w^{\perp\perp} = H$. 
\end{proof}

\begin{example}\label{ex:theonlyone}
The reverse implication from Theorem~\ref{t:bapolars} does not hold. In fact, we exhibit an $\ell$-group $H$ for which $\Min{H} = \Qin{H}$ is not compact. Let $H$ be the $\ell$-group
\begin{equation}\nonumber
H \coloneqq \{\,  f \colon \N \to \Z \mid {\rm supp}(f) \text{ is finite} \,\},
\end{equation} 
\noindent where ${\rm supp}(f) \coloneqq \{\, n \in \N \mid f(n) \ne 0 \,\}$, with coordinate-wise operations, and the function $\overline{0}$ constantly equal to $0$ as the group identity. We show that $\Qin{H} = \Min{H} = (\N, \tau_d)$, where $\tau_d$ is the discrete topology.
\begin{claim}\label{cl:example}
The distributive lattice $\Conp{H}$ is $\Polp{H}$, and it is isomorphic to $(\N^*, \cap, \cup)$, where 
\[
\N^* \coloneqq \{\, S \sst \N \mid S \text{ is finite} \,\}.
\]
\end{claim}
\begin{proof}
We start from the latter, and observe that for any $f \in H^+$, 
\begin{equation}\nonumber
f \wedge g = \overline{0}\ \ \text{ if, and only if, }\ \ {\rm supp}(f) \cap {\rm supp}(g) = \emptyset.
\end{equation}
Therefore, $h \wedge g = \overline{0}$ for every $g \in f^{\perp}$ precisely when ${\rm supp}(h) \sst {\rm supp}(f)$. From which we conclude
\begin{equation}\nonumber
f^{\perp\perp} = \{\, h \in H \mid {\rm supp}(h) \sst {\rm supp}(f) \,\}.
\end{equation}
Thus, from \eqref{eq:prinpol1}--\eqref{eq:prinpol2}, the map $f^{\perp\perp} \mapsto {\rm supp}(f)$ is a lattice isomorphism $\Polp{H} \cong (\N^*, \cap, \cup)$. Moreover, 
\begin{equation}\nonumber
\Ksg{(f)} = \{\, g \in H \mid |g| \leq f^n\text{ for some }n \in \N \,\}.
\end{equation}
Hence, every positive element of $\Ksg{(f)}$ has support included in ${\rm supp}(f)$. Conversely, if $g \in H^+$ and ${\rm supp}(g) \sst {\rm supp}(f)$, then $g(n) \geq 0$ implies $f(n) \geq 0$. Now, since the functions have finite support, it is possible to find $m \in \N$ so that $g(n) \leq mf(n)$, for every $n \in \N$. Therefore, $g \in \Ksg{(f)}$.
\end{proof}
\noindent From Lemma~\ref{l:queen} and Theorem~\ref{t:specdual}, we conclude $\Spec{H} = \Qin{H}$.
\begin{claim}
The Stone dual space of $(\N^*, \cap, \cup)$ is $(\N, \tau_d)$.
\end{claim}
\begin{proof}
It is straightforward that $(\N, \tau_d)$ is a generalised spectral space whose compact opens are precisely the finite subsets of $\N$. The result now follows from~\cite[Theorem 15]{Stone38}.
\end{proof}
\noindent Therefore, $\Spec{H} = \Qin{H} = (\N, \tau_d)$ is $\Min{H}$, since the specialisation order of $(\N, \tau_d)$ is trivial. This completes Example~\ref{ex:theonlyone}.
\end{example}

\begin{remark}[Compactness of minimal spectra]The equivalence of  (1), (2), and (3) in Theorem~\ref{t:bapolars} is a well-known result, of which we have provided a streamlined proof for the reader's convenience. There is a substantial literature concerned with the compactness of minimal spectra of various structures, and we cannot do  justice to it here. In connection with Theorem~\ref{t:bapolars} we ought to at least mention  Speed's paper~\cite{Speed1969} for distributive lattices, and Conrad's and Martinez' paper~\cite{ConradMartinez1990} for $\ell$-groups. Let us also mention that, in the Archimedean case, compactifications of minimal spectra of lattice-groups were recently shown to be inextricably related to the construction of projectable hulls, see~\cite{BMMP2018, HagerMcGovern}.
\end{remark}

\begin{corollary}\label{c:cbapolars}
For any partially ordered group $G$ and any variety $\V$ of $\ell$\nbd{-}groups, the minimal layer of $\Pord{G}{\V}$ is compact if, and only if, any one of the equivalent conditions of Theorem~\ref{t:bapolars} holds for $\Free{G}{\V}$.
\end{corollary}

\begin{proof}
Combine Theorem~\ref{t:main} and Theorem~\ref{t:bapolars}.
\end{proof}

For a group $G$, set 
\begin{equation}\nonumber
\Xa{(a)} \coloneqq \{\, S \sst G \mid a \in S \,\}\ \ \text{ and }\ \ \Xa^{c}{(a)} \coloneqq \{\, S \sst G \mid a \not \in S \,\}, \quad \text{ for }a \in G.
\end{equation}
We endow $2^{G}$ with the smallest topology containing all sets $\Xa{(a)}$ and $\Xa^{c}{(a)}$. With this topology, the space $2^{G}$ is easily shown to be a Stone space, see~\cite[p.\ 6]{CR2016}. Assume additionally that $G$ is partially ordered. It is then elementary that the set of pre-cones $C \sst G$ such that $C \cap \iv{C} = \{\, \e \,\}$, with the subspace topology inherited from $2^{G}$, is homeomorphic to the subspace~$\Rord{(G)}$ of $\RPord{(G)}$ consisting of right orders on $G$ topologised as in \eqref{eq:rotop}. This homeomorphism restricts to one between normal pre-cones $C \sst G$ with the property $C \cap \iv{C} = \{\, \e \,\}$, and the space $\Ord{(G)}$ of orders on $G$.

\begin{lemma}\label{l:compactrord}
Suppose $\V$ is the variety of all $\ell$-groups (respectively, representable or Abelian $\ell$-groups). For any partially ordered group $G$, the minimal layer of $\Pord{G}{\V}$ is compact. 
\end{lemma}

\begin{proof}
By Theorem~\ref{t:minorders}, if $\V$ is the variety of all $\ell$-groups (respectively, representable or Abelian $\ell$\nbd{-}groups), the minimal layer of $\Pord{G}{\V}$ is the space of right orders (resp., orders) on $\eta[G]$. We can now  conclude, since $\Rord{(\eta[G])}$ (resp., $\Ord{(\eta[G])}$) is a closed subspace of $2^{\eta[G]}$~\cite[Problem~1.38]{CR2016}.
\end{proof}

\begin{theorem}\label{t:compactin3vars}
Suppose $\V$ is the variety of all $\ell$-groups (respectively, representable or Abelian $\ell$-groups). For any partially ordered group $G$, the minimal spectrum $\Min{\Free{G}{\V}}$ is compact.
\end{theorem}

\begin{proof}
Combine Theorem~\ref{t:main} and Lemma~\ref{l:compactrord}.
\end{proof}

\begin{theorem}\label{t:asbooleanspaces}
Suppose $\V$ is the variety of all $\ell$-groups (respectively, representable or Abelian $\ell$-groups). For any partially ordered group $G$, the minimal layer of $\Pord{G}{\V}$ is a Stone space with dual Boolean algebra $\Polp{\Free{G}{\V}}$.
\end{theorem}

\begin{proof}
Combine Corollary~\ref{c:cbapolars} and Lemma~\ref{l:compactrord}.
\end{proof}

\section{Representable varieties and pre-orders}\label{s:rep}

It is classical that the set $\Conn{H}$ of ideals of $H$ ordered by inclusion is a complete sublattice of $\Con{H}$, and hence of the lattice of subgroups of~$H$~\cite[Theorem 8.7]{Darnel95}. We write $\Ideal{(S)}$ to denote the ideal generated by $S \sst H$. If $x \in H$, we write $\Ideal(x)$ for the {\em principal ideal} $\Ideal(\{\, x \,\})$ generated by $\{\, x \,\}$, and $\Connp{H}$ for the collection of all the principal ideals of $H$. 

\begin{proposition}\label{p:convgenerated}
For any $\ell$-group $H$, the convex $\ell$-subgroup $\Ksg{(S)}$ generated by a subset $S \sst H$ is the set
\begin{equation}
\{\, x \in H \mid |x| \leq s\text{ for some }s \in \langle |S| \rangle \,\},
\end{equation}
where $|S| = \{\, |y| \mid y \in S \,\}$, and $\langle T \rangle$ is the monoid generated by a subset $T$ of $H$.
\end{proposition}

\begin{proof}
Set 
\[
K \coloneqq \{\, x \in H \mid |x| \leq s\text{ for some }s \in \langle |S| \rangle \,\}.
\]
It is immediate that $K \sst \Ksg{(S)}$. For the other inclusion, we show that $K$ is a convex $\ell$-subgroup containing $S$. First of all, $S \sst K$, and $K$ is a subgroup of~$H$, as $|x| \leq s$ and $|y| \leq t$ implies $|xy| \leq |x||y||x| \leq sts$. (Recall $|xy| \leq |x||y||x|$ is an $\ell$-group law, see~\cite[Theorem 4.13]{Darnel95}.) To show that $K$ is a sublattice, observe that if $x, y \in K$, then $x \vee \iv{x} \leq s$ and $y \vee \iv{y} \leq t$ for $s, t \in \langle |S| \rangle$ or, equivalently,  $x \leq s$, $\iv{x} \leq s$, $y \leq t$, and $\iv{y} \leq t$. Therefore, $x\iv{y} \leq st$, and $y\iv{x} \leq ts$. Now, since $\e \leq s,t$, the latter entails 
\[
|x\iv{y}| = x\iv{y} \vee y\iv{x} \leq sts \in \langle |S| \rangle.
\] 
Thus, $x\iv{y} \in K$ and hence, also $x\iv{y} \vee \e \in K$, since 
\[
\e \leq x\iv{y} \vee \e \leq |x\iv{y}| \leq sts.
\]
Therefore, $x \vee y = (x\iv{y} \vee \e)y \in K$. We can now conclude, since 
\[
x \wedge y = \iv{(\iv{x} \vee \iv{y})} \in K.
\] 
It remains to show that $K$ is convex. For this, pick $x, y \in K$, and $x \leq z \leq y$ for some $z \in H$. But then, $\e \leq z\iv{x} \leq y\iv{x}$ and hence, $|z\iv{x}| \leq |y\iv{x}| \leq s$, for some $s \in \langle |S| \rangle$. Therefore, $z\iv{x} \in K$, and $(z\iv{x})x = z \in K$.
\end{proof}

We write $N(S)$ to denote the {\em normal closure} $\{\, \iv{y}xy \mid y \in H, x \in S \,\}$ of $S$.

\begin{proposition}\label{p:idealgenerated}
For any $\ell$-group $H$, the ideal $\Ideal{(S)}$ generated by a subset $S \sst H$ is the convex $\ell$-subgroup generated by the normal closure $N(S)$ of $S$; equivalently, $\Ideal{(S)} = \Ksg{(N(S))}$.
\end{proposition}

\begin{proof}
By Proposition~\ref{p:convgenerated}, we know $\Ksg{(N(S))}=\Ksg{(\langle |N(S)|\rangle)}$, or equivalently, $\Ksg{(N(S))}=\Ksg{(\langle N(|S|)\rangle)}$. Clearly, $\Ksg{(N(S))} \sst \Ideal{(S)}$. To conclude, it suffices to show that $\Ksg{(N(S))}$ is normal. But this is immediate, since 
\[
|x| \leq \prod_{I} \iv{y}_is_iy_i,
\] 
for $s_i \in |S|$, $y_i \in H$, and $I$ finite, implies 
\[
|\iv{z}xz| = \iv{z}|x|z \leq \iv{z}\big{(}\prod_{I} \iv{y}_is_iy_i\big{)}z = \prod_{I} \iv{z}\iv{y}_is_iy_iz \in \langle N(|S|)\rangle,
\]
and hence, $\iv{z}xz \in \Ksg{(N(S))}$ for every $z \in H$.
\end{proof}

\begin{proposition}\label{p:princcong}
For any $\ell$-group $H$, and for any $x, y \in H^+$, $z \in H$, the following hold.
\begin{enumerate}
\item [{\rm (1)}] $\Ideal(z) = \Ideal(|z|) = \{\, h \in H \mid |h| \leq \prod_{i \in I} \iv{w}_i|z|w_i, \text{ for finitely many }w_i \in H \,\}$. 
\item [{\rm (2)}] $\Ideal(x \wedge y) \sst \Ideal(x) \cap \Ideal(y)$ and $\Ideal(x \vee y) = \Ideal(x) \vee \Ideal(y)$. 
\end{enumerate}
\end{proposition}

\begin{proof}
(1) is immediate from Proposition~\ref{p:idealgenerated}. Item (2) follows from the fact that $\ksg \in \Con{H}$ contains $x, y \in H^+$ if, and only if, $x \vee y \in \ksg$; also, $x \wedge y$ is contained in every convex $\ell$-subgroup containing $x$ or $y$.
\end{proof}

\noindent In light of Proposition~\ref{p:princcong}, when the set $\Connp{H}$ of principal ideals of any $\ell$-group $H$ is ordered by inclusion, it is a join\nbd{-}semilattice with minimum $\Ideal(\e) = \{\, \e \,\}$, and a subsemilattice of $\Conn{H}$ (and therefore of $\Con{H}$).

\begin{proposition}\label{p:polnorm}
For any $\ell$-group $H$, the following are equivalent.
\begin{enumerate}
\item $H$ is representable.
\item Each polar is normal.
\end{enumerate}
\end{proposition}

\begin{proof}
See, e.g.,\ \cite[Proposition 47.1]{Darnel95}.
\end{proof}

\begin{remark}
Lemma~\ref{l:queen} and Proposition~\ref{p:polnorm} ensure $\Qin{H} \sst \Specn{H}$ for every representable $\ell$-group $H$.
\end{remark}

\begin{lemma}\label{l:reprcharcong}
For any $\ell$-group $H$, the following are equivalent.
\begin{enumerate}
\item $H$ is representable.
\item For every $x, y \in H^+$, $\Ideal{(x)} \cap \Ideal{(y)} \sst \Ideal{(x \wedge y)}$.
\end{enumerate}
\end{lemma}

\begin{proof}
For (1) $\Rightarrow$ (2), we use the characterising property of representable $\ell$\nbd{-}groups stated in Proposition~\ref{p:polnorm}: we show that $\Ideal{(x)} \cap \Ideal{(y)} \sst \Ideal{(x \wedge y)}$ entails that $x^{\perp}$ is normal, for every $x \in H$. Assume $x \wedge y = \e$, for some $x, y \in H^+$, that is, assume $y \in x^{\perp}$ for $x, y \in H^+$. Thus, 
\begin{equation}
\Ideal{(x)} \cap \Ideal{(y)} = \Ideal{(x \wedge y)} = \{\, \e \,\}.
\end{equation}
This means that $x \wedge \iv{a}ya = \e$ for every $a \in H$, as $(x \wedge \iv{a}ya) \in \Ideal{(x)} \cap \Ideal{(y)} = \{\, \e \,\}$ for every $a \in H$. Thus, $\iv{a}ya \in x^{\perp}$, for every $a \in H$. 

For (2) $\Rightarrow$ (1), we show that $\Ideal{(x)} \cap \Ideal{(y)} \sst \Ideal{(x \wedge y)}$ for every $x, y$ in the positive cone of a representable $\ell$-group $H$. For this, let $H$ be a subdirect product of $\prod_{t \in T} C_t$, for $C_t$ totally ordered groups. If $|z| \in \Ideal{(x)} \cap \Ideal{(y)}$, Proposition~\ref{p:idealgenerated} entails 
\[
|z| \leq \prod_{I} \iv{a}_ixa_i,\text{ and }|z| \leq \prod_{J} \iv{b}_jyb_j.
\]
Thus, in a given factor $C_t$ of the product $\prod_{t \in T} C_t$,
\[
(|z|)_t \leq (\prod_{I} \iv{a}_ixa_i)_t,\text{ and }(|z|)_t \leq (\prod_{J} \iv{b}_jyb_j)_t.
\]
Since the group operation is also defined coordinate-wise, we obtain 
\[
(|z|)_t \leq \prod_{I} (\iv{a}_i)_t(x)_t(a_i)_t,\text{ and }(|z|)_t \leq \prod_{J} (\iv{b}_j)_t(y)_t(b_j)_t.
\]
Without loss of generality, we can assume $(x)_t \leq (y)_t$, and hence
\[
(|z|)_t \leq \prod_{I} (\iv{a}_i)_t(x)_t(a_i)_t = \prod_{I} (\iv{a}_i)_t(x \wedge y)_t(a_i)_t.
\]
Since $t \in T$ was arbitrary, $|z| \leq \prod_{I} \iv{c}_i(x \wedge y)c_i$, for some $c_i \in H$, and $I$ finite. Thus, $z \in \Ideal{(x \wedge y)}$.
\end{proof}

Consider the map 
\begin{equation}\label{eq:ontoslh}
\Conp{H} \xrightarrow{ \ g \ } \Conn{H}
\end{equation}
defined by $\Ksg{(x)} \mapsto \Ideal{(x)}$, for $x \in H^+$. 

\begin{theorem}\label{t:latticehomorep}
For any $\ell$-group $H$, the map $g \colon \Conp{H} \to \Conn{H}$ defined in \eqref{eq:ontoslh} is a join\nbd{-}semilattice homomorphism preserving minimum such that 
\[
g[\Conp{H}]=\Connp{H}.
\] 
Moreover, it is a lattice homomorphism if, and only if, $H$ is representable.
\end{theorem}

\begin{proof}
First, observe that from $\Ksg{(x)} = \Ksg{(y)}$, it follows that $x \in \Ksg{(y)} \sst \Ideal{(y)}$ and $y \in \Ksg{(x)} \sst \Ideal{(x)}$. Thus, if $\Ksg{(x)} = \Ksg{(y)}$, then $\Ideal{(x)} = \Ideal{(y)}$. Hence, $g$ is well-defined. Further, it is clear that $g[\Conp{H}]=\Connp{H}$. By Proposition~\ref{p:princcon}.(2) and Proposition~\ref{p:princcong}.(2), 
\[
g(\Ksg{(x)} \vee \Ksg{(y)}) = g(\Ksg{(x \vee y)}) = \Ideal{(x \vee y)} = \Ideal{(x)} \vee \Ideal{(y)}.
\]
Finally, $\Ksg{(\e)} = \{\, \e \,\} = \Ideal{(\e)}$.

Since $g(\Ksg{(x)} \cap \Ksg{(y)}) =  \Ideal{(x)} \wedge \Ideal{(y)}$ if, and only if, $\Ideal{(x \wedge y)} = \Ideal{(x)} \cap \Ideal{(y)}$, the second statement follows from Lemma~\ref{l:reprcharcong}, Proposition~\ref{p:princcon}.(2), and Proposition~\ref{p:princcong}.(2).
\end{proof}

\begin{proposition}\label{p:compideal}
For any $\ell$-group $H$, the set $\Connp{H}$ consists precisely of the compact elements of $\Conn{H}$. 
\end{proposition}

\begin{proof}
Pick a compact element $\ksg$ of $\Conn{H}$, and note that $\ksg \sst \bigvee_{x \in \ksg}\Ideal{(x)}$. By compactness, also $\ksg \sst \Ideal{(x_1)} \vee \dots \vee \Ideal{(x_n)}$, for $x_1, \dots, x_n \in \ksg$. Hence, 
\[
\ksg = \Ideal{(x_1 \vee \dots \vee x_n)}
\]by Proposition~\ref{p:idealgenerated}.(2). Conversely, consider $\Ideal{(x)} \sst \bigvee_{J} \ksg_j$ for some $x \in H$, and $\ksg_j \in \Conn{H}$. But then, the element $|x| \in H^+$ equals $y_1\cdots y_m$, for some $y_1,\dots,y_m \in\bigcup_{J} \ksg_j$. Thus, $\Ideal{(x)} \sst \ksg_{j_1} \vee \dots \vee \ksg_{j_m}$, where $j_1, \dots, j_m \in J$ are such that $y_i \in \ksg_{j_i}$.
\end{proof}

\begin{corollary}\label{c:specngenspec}
For any representable $\ell$-group $H$, $\Specn{H}$ is homeomorphic to the dual space of the distributive lattice with minimum $\Connp{H}$. Hence, $\Specn{H}$ is generalised spectral. 
\end{corollary}

\begin{proof}
For any representable $\ell$-group $H$, set $D \coloneqq \Connp{H}$. Then, the map
\begin{align}\nonumber
X(D) & \xrightarrow{ \  \ }  \Conn{H} \\\nonumber
\jd & \xmapsto{ \hskip .35 cm}  \bigvee\{\, \Ideal(x) \mid \Ideal(x) \in \jd \,\}
\end{align}
restricts to a homeomorphism between $X(D)$ and $\Specn{H}$. The proof goes along the same lines as the one of Theorem~\ref{t:specdual}, and it is based on Proposition~\ref{p:openbase}, Proposition~\ref{p:princcong}, Theorem~\ref{t:latticehomorep}, and Proposition~\ref{p:compideal}.
\end{proof}

For any $\ell$-group $H$, we consider the function
\begin{align}
\Con{H} & \xrightarrow{ \ \nu \ }  \Con{H} \\\nonumber
\ksg & \xmapsto{ \hskip .35 cm}  \bigcap_{x \in H} \iv{x}\ksg x,
\end{align}
that maps any convex $\ell$-subgroup $\ksg$ of $H$ to the largest ideal contained in $\ksg$. (For a related use of this map, compare the characterisation of representable $\ell$-groups in~\cite[Theorem 2.4.4.(d)]{Steinberg2010}.) Recall that an endofunction $\iota\colon S\to S$ on a partially ordered set $S$ is an \emph{interior operator} if  it is contracting ($\iota(x) \leq x$), monotone ($x\leq y$ entails $\iota(x)\leq \iota(y)$), and idempotent ($\iota\circ\iota$ coincides with $\iota$ on~$S$). The fixed points of $\iota$ are called the \emph{open elements} of $S$. The map $\nu$ is an interior operator on $\Con{H}$, and  $\Conn{H}$ consists precisely of the open elements of $\nu$. 

\begin{lemma}\label{l:retraction}
For any $\ell$-group $H$, the map $\nu$ descends to an interior operator $\nu$ on $\Spec{H}$ such that $\Specn{H}$ consists precisely of the open elements of $\nu$ if, and only if, $H$ is representable. In this case, $\nu \colon \Spec{H} \to \Specn{H}$ is a continuous retraction.
\end{lemma}  

\begin{proof}
First, if $\nu$ is an interior operator onto $\Specn{H}$, since $\nu(\m) \sst \m$ for every $\m \in \Min{H}$, we can conclude $\nu(\m) = \m$ and hence, every minimal prime is an ideal. Thus, by Proposition~\ref{p:representable}, $H$ is representable. Conversely, suppose that $H$ is representable, and take $x, y \in H$ such that $x \land y = \e$. We show that either $x \in \nu(\p)$ or $y \in \nu(\p)$, for every $\p \in \Spec{H}$. If $x \not \in \nu(\p)$, there exists a $w \in H$ such that $\iv{w}xw \not \in \p$. Now, since $x \land y = \e$ and $H$ is representable, by Proposition~\ref{p:polnorm} also $\iv{w}xw \land \iv{z}yz = \e$ for every $z \in H$. Therefore, since $\p$ is prime, $\iv{z}yz \in \p$ for every $z \in H$, that is, $y \in \nu(\p)$. The fact that $\nu$ descends to an interior operator $\nu \colon \Spec{H} \to \Specn{H}$ is immediate. Also, $\nu(\p) = \p$ if, and only if, $\p \in \Specn{H}$. Finally, observe that 
\[
\iv{\nu}[\Supp^*{(x)}] = \bigcup_{y \in H} \Supp{(\iv{y}xy)}
\] 
and hence, $\nu$ is continuous.
\end{proof}

\begin{remark}
The map $\nu \colon \Spec{H} \to \Specn{H}$ sends a prime subgroup $\p$ to the kernel $\ker{R_\p}$ of the map $R_\p \colon H \to \Aut{H/\p}$ defined in \eqref{eq:Rrrc}.
\end{remark}

For any partially ordered group $G$ and any variety $\V$ of representable $\ell$-groups, consider 
\begin{align}
\Pord{G}{\V} & \xrightarrow{ \ \beta \ }  \Pord{G}{\V} \\\nonumber
C & \xmapsto{ \hskip .35 cm}  \bigcap_{t \in G} \iv{t}C t.
\end{align} 

\begin{theorem}\label{t:main_rep}
For any partially ordered group $G$ and any variety $\V$ of representable $\ell$-groups, the following hold.
\begin{enumerate}
\item The set $\Bord{G}{\V}$ with the subspace topology induced from $\Pord{G}{\V}$ is generalised spectral. 
\item The map $\beta$ is an interior operator on $\Pord{G}{\V}$ such that $\Bord{G}{\V}$ consists precisely of the open elements of $\beta$. Moreover, $\beta \colon \Pord{G}{\V} \to \Bord{G}{\V}$ is a continuous retraction.
\end{enumerate}
\end{theorem}

\begin{proof}
(1) follows immediately from Theorem~\ref{t:main} and Corollary~\ref{c:specngenspec}.

For (2), observe that for a right pre-order $C$, the set $\beta(C)$ is clearly a normal submonoid of $G$. Moreover, if $a \not \in \beta(C)$, there is a conjugate $\iv{t}at \not\in C$ for $t \in G$; hence, by Theorem~\ref{t:reptrans2}, $\iv{s}as \in \iv{C}$ for each $s \in G$, that is, $\iv{a} \in \beta(C)$. Therefore, $\beta(C) \cup \iv{\beta(C)} = G$, and $\beta$ is a well-defined function onto $\Bord{G}{\V}$. Further, $\beta$ sends a right pre-order $C$ in $\Pord{G}{\V}$ to the largest pre-order contained in $C$, and $\beta(C) = \pi(\nu(\kappa(C)))$. Thus, an application of Theorem~\ref{t:main} and Lemma~\ref{l:retraction} completes the proof.
\end{proof}

\section*{Acknowledgements}
\noindent The research of the first-named author was supported by the Swiss National Science Foundation grant 200021\_165850.

\end{document}